\documentclass[a4paper,11pt]{elsarticle}

\usepackage{amscd,amsmath,amsthm,amsfonts}
\usepackage{graphicx}
\usepackage{amssymb,color,amsbsy}%, amsthm}
\usepackage[matrix,arrow]{xy}
\usepackage{mathabx}
\usepackage{tikz}
\usetikzlibrary{decorations.pathmorphing}
\usepackage{caption}
\usepackage{amsthm}

\usepackage[pagebackref=true, bookmarksopen=true,colorlinks=true, linkcolor=red,citecolor=blue]{hyperref}

\usepackage{tikz-cd}

\usepackage{tkz-graph}
\usetikzlibrary{arrows}

\usepackage{multicol}

\usepackage{subcaption}

\newtheorem{theorem}{Theorem}[section]

\newtheorem{corollary}[theorem]{Corollary}

\newtheorem{example}[theorem]{Example}

\newtheorem{lemma}[theorem]{Lemma}

\newtheorem{proposition}[theorem]{Proposition}

\newcommand{\modulo} [2] {{#1} \operatorname{mod} {#2} }

\newcommand{\Shift}{\operatorname{Shift}}
\newcommand{\cycle}{\operatorname{cycle}}
\newcommand{\pos}{\operatorname{pos}}
\newcommand{\Circulante}{\operatorname{Circ}}

\newcommand{\Ind}[1]{\operatorname{ind}(#1)}
\newcommand{\rank}{\operatorname{rank}}
\newcommand{\dnull}[1]{\operatorname{null}(#1)}
\newcommand{\Null}[1]{\operatorname{Null}\left(#1\right)}
\newcommand{\Rank}[1]{\operatorname{Rank}(#1)}

\newcommand{\PCirc}[5]{#4 P_{#1}^{#2}+ #5 P_{#1}^{#3}}
\newcommand{\PCircN}[4]{#3 I_{#1} + #4 P_{#1}^{#2}}

\newcommand{\perm}{\operatorname{perm}}

\begin{document}
	
	\begin{abstract}
		This work presents closed formulas for determinant,  permanent,  inverse, and Drazin inverse of circulant matrices with two non-zero coefficients.
	\end{abstract} 


\begin{frontmatter}
	\title{Using digraphs to compute determinant, permanent and Drazin inverse of circulant matrices with two parameters}
	
	\author[ame]{Andr\'es M. Encinas} %\corref{cor1}}%\fnref{fn1}}
\ead{andres.marcos.encinas@upc.edu}

\author[daj]{Daniel A. Jaume}%\corref{cor1}}%\fnref{fn1}}
\ead{djaume@unsl.edu.ar}

%\author[daj]{Adri\'{a}n Pastine}%\fnref{fn1}}
%\ead{adrian.pastine.tag@gmail.com}

\author[daj]{Cristian Panelo}%\corref{cor1}}%\fnref{fn2}}
\ead{crpanelo@unsl.edu.ar}

\author[dev]{Denis E.\@ Videla}
% -- CIEM, Universidad Nacional de C\'ordoba, CONICET, FAMAF (5000) C\'ordoba, Argentina. 
\ead{devidela@famaf.unc.edu.ar}

%\cortext[cor1]{Corresponding author: Daniel A. Jaume}
%\cortext[cor2]{Principal corresponding author}
%\fntext[fn1]{This is the specimen author footnote.}
%\fntext[fn2]{Another author footnote, but a little more longer.}
%\fntext[fn3]{Yet another author footnote. Indeed, you can have
% 	any number of author footnotes.}

\address[daj]{Departamento de Matem\'{a}ticas. 
	Facultad de Ciencias F\'{\i}sico-Matem\'{a}ticas y Naturales. Universidad Nacional de San Luis. San Luis, Argentina} 
%1er. piso, Bloque II, Oficina 54. 	 	Ej\'{e}rcito de los Andes 950. 
%San Luis, Rep\'{u}blica Argentina. 
%D5700HHW.}
%\adress[enc]{Departament de Matem\`{a}tiques. Universitat Polit\`{e}cnica de Catalunya.}
\address[dev]{FaMAF - CIEM (CONICET), Universidad Nacional de C\'ordoba.}
\address[ame]{Departament de Matem\`atiques. Universitat Polit\`ecnica de Catalunya. Barcelona, Spain.}

\date{Received: date / Accepted: date}
\end{frontmatter}

\section{Introduction}

Circulant matrices appear in many applications, for example, to approximate the finite difference of elliptic equations with periodic boundary, and to approximate periodic functions with splines. The circulant matrices play an important role in coding theory and statistic. The standard reference is \cite{davis2012circulant}.

Among the main problems associated with circulant matrices are those of determining invertibility conditions and computing their Drazin inverse. These problems have been widely treated in the literature by using the primitive  $n$--th root of unity and some polynomial associated with the circulant matrices,  see \cite{davis2012circulant}, \cite{carmona2015inverses}, and  \cite{shen2011determinants}. There exist some classical and well-known results that enable us to solve almost everything we would raise about the inverse or Drazin inverse of circulant matrices. Nevertheless, when we deal with specific families of circulant matrices these classical results give us unmanageable formulas. Therefore, it has interesting to look for alternative descriptions, and in fact, there exist many papers devoted to this question. The direct computation for the inverse of some circulant matrices have been proposed in many works, see for example \cite{fuyong2011inverse},  \cite{shen2011determinants}, \cite{yazlik2013inverse}, \cite{jiang2014invertibility}, \cite{carmona2015inverses}, and \cite{radivcic2016k} (in chronological order).

Besides, the combinatorial structure of circulant matrices also has deserved attention. The graphs whose adjacency matrix is circulant had been study in many works, specially those with integral spectrum, see for example \cite{so2006integral}, \cite{bavsic2009clique}, \cite{ilic2011new}, \cite{petkovic2011further}, and  \cite{sander2018kernel}.

In this work, we delve into the combinatorial structure of circulating matrices with only two non-null generators, by considering the digraphs associated with this kind of matrices. Therefore, we complete a previous work of some of the authors, see  \cite{encinas2019Circulant}, where only a specific class of these matrices was considered.

We use digraphs in the present work, for all of graph-theoretic notions not explicitly defined here, the reader is referred to~\cite{bang2008digraphs}. 

We warn the reader that, as is usual in circulant matrices, our matrices indexes and permutations start at zero. Hence, permutations in this work are bijections over \(\{0,\dots, n-1\}\), and if 
\[
A=\left[
\begin{matrix}
1 & 2 \\
3 & 4 
\end{matrix}
\right],
\]
then the set of indexes of \(A\) is \(\{0,1\}\). So \(A_{0\,0}=1\) and \(A_{1\,0}=3\). This is also why in the present work \([n]\) denotes the set \(\{0,\dots, n-1\}\) instead of \(\{1,\dots, n\}\). 
 
For permutations, we use the cyclic notation: 
\((0\;\;2\;\;1\;\;4)(5\;\;6)\) is the permutation (in row notation) \((2\;\;4\;\;1\;\;3\;\;0\;\;6\;\;5)\), see \cite{rotman2012introduction}. Given a permutation \(\alpha\) of \([n]\), we denote by \(P_{\alpha}\) the \(n\times n\) matrix defined by \((P_{\alpha})_{\alpha(j)\,j}=1\) and \(0\) otherwise. The matrix \(P_{\alpha}\) is known as the permutation matrix associated to \(\alpha\). It is well-known that \(P_{\alpha}^{-1}=P_{\alpha}^{T}\). The assignation $\alpha\mapsto P_{\alpha}$ from the Symmetric Group  $\mathbb{S}_n$ to the General Lineal Group $\mathrm{GL}(n)$ is a representation of $\mathbb{S}_n$, i.e., $P_{\alpha\beta}=P_{\alpha}P_{\beta}$ where product is composition. % Other interesting fact that we use is: \(P_{\alpha} e_{j}=e_{\alpha(j)}\), where \(e_{j}\) is the canonical vector whose unique one is in the $j$-th coordinate and $0$ otherwise. 

The cycle type of a permutation \(\alpha\) is an expression of the form \[\left(1^{m_{1}}, 2^{m_{2}},\dots, n^{m_{n}} \right),\] where \(m_{k}\) is the number of cycles of length \(k\) in \(\alpha\). It is well-known that the conjugacy classes of \(\mathbb{S}_{n}\) are determined by the cycle type, see page 3 of \cite{sagan2013symmetric}. Thus, \(\alpha\) and \(\beta\) are conjugated (i.e. there exist a permutation \(\sigma\) such that \(\sigma\alpha \sigma^{-1}=\beta\)) if and only if \(\alpha\) and \(\beta\) have the same cycle type.
   
We use the matrix associated to the permutation
\[\tau_{n}=(n-1 \;\;\; n-2 \; \cdots \;2 \;\;1 \;\;0)\]
many times along this work, so instead of \(P_{\tau_{n}}\) we just write \(P_{n}\).

Notice that, for \(k\in \mathbb{Z}\), we have that \(P_{n}^{0}=I_{n}\), \(P_{n}^{k}= P_{n}^{\modulo{k}{n}}=P_{\tau_{n}^{k}}\), \(\det(P_{n})=(-1)^{n-1}\), and  \(\left(P_{n}^{k}\right)^{-1}=P_{n}^{n-k}\). Moreover \(\tau_{n}^{k}(i)=\modulo{i-k}{n}\).

A matrix \(C=(c_{i\, j})\) is named circulant with parameters \(c_{0}, c_{1}, \dots , c_{n-1}\) if
\[
C =
\left[ 
\begin{array}{cccc}
	c_{0} & c_{1} & \dots & c_{n-1} \\
	c_{n-1} & c_{0} & \dots & c_{n-2} \\
	\vdots & \vdots & \ddots & \vdots \\
	c_{1} & c_{2} & \dots & c_{0}
\end{array}
\right] = \Circulante (c_{0},\dots,c_{n-1})
\]
or equivalently
\[
c_{i\, j} = c_{\modulo{j-i}{n}}.
\]
We have that
\[
\Circulante (c_{0},\dots,c_{n-1}) = c_{0}I_{n}+\cdots+c_{n-1}P_{n}^{n-1},
\]
the numbers \(c_{0},\dots,c_{n-1}\) are called the parameters of the circulant matrix \(C\).

Let \(A\) be an \(m\times n\) matrix and let \(S\subseteq [m]\), \(T\subseteq [n]\). The submatrix of $A$ obtained by deleting the rows in \(S\) and the columns in \(T\) is denoted by \(A(S|T)\).
    
The paper is organized as follows. In Section \ref{subsec_Untangling}, we present our main idea about how to work with circulant matrices with just two non-zero parameters: untangling the associated digraphs. In Section \ref{Sec_Sigma}, we explicitly find the matrices that untangle the digraphs associated with our matrices. In Section \ref{Sec_Det_Perm_aP_bP}, we find explicit formulas for determinant and permanent of circulant matrices with two non-zero parameters. In Section \ref{Sec_Inv_aP_bP}, we give an explicit formula for the inverse of non-singular circulant matrices with two non-zero coefficients. In Section \ref{Sec_Inv_Dra_aP_bP}, we give an explicit formula for the Drazin inverse of singular circulant matrices with two non-zero parameters. Finally, in Sections \ref{Sec_Blck_Circ}, \ref{Sec_Blck_Circ_Det}, and \ref{Sec_Blck_Circ_Det_Inv} we generalized the previous results for block circulant matrices.
%
%%%%%%%%%%%%%%%%%%%%%%%%%%%%%%%%%%%%%%%%%%%%%%%%%%%%%%%%%%%%%%%%%%%%%%%%%

%\section{Circulant matrices with two nonzero parameters}

\medskip

%%%%%%%%%%%%%%%%%%%%%%%%%%%%%%%%%%%%%%%%%%%%%%%%%%%%%%%%%%%%%%%%%%%%%%%%%%%%%%%
%
%%%%%%
%
%%%%%%%%%%%%%%%%%%%%%%%%%%%%%%%%%%%%%%%%%%%%%%%%%%%%%%%%%%%%%%%%%%%%%%%%%%%%%%%

\section{Untangling the skein: two key permutations} \label{subsec_Untangling}

Let $n,s_{1},s_{2}$ and \(s\) be non-negative integers such that $0\leq s_{1} < s_{2} < n$ and \(0 <s<n\). Let $a,b$ be non-zero complex numbers. We are going to work with the following type of \(n\times n\) circulant matrices:
\begin{equation}
\PCirc{n}{s_{1}}{s_{2}}{a}{b}.
\end{equation}
We call it a circulant matrix with two parameters.

There are \(\binom{n}{2}\) forms of circulant matrices with two parameters. Since 
\(\PCirc{n}{s_{1}}{s_{2}}{a}{b}=P_{n}^{s_{1}} \left( \PCircN{n}{s_{2}-s_{1}}{a}{b}\right)\) and 
\( \PCircN{n}{n-s_{1}}{a}{b}=\left( \PCircN{n}{s_{1}}{a}{b}\right)^{T} \), 
it looks like we only need to understand \((n-1)/2\) of them. As we will see later, it is enough with just one of them: \(aI_{n}+bP_{n}\).

As usual in number theory, the greatest common divisor of two integers 
\(n\) and \(s\) is denoted by \((n,s)\). We find useful the following notation
\[
n\backslash s:=\tfrac{n}{(n,s)}.
\] 
We read it ``n without s''. Notice that \(s \, (n\backslash s) = n \, (s \backslash n) = [n,s]\), where \([n,s]\) is the lowest common multiple of \(n\) and \(s\). The rest of the integer division of \(x\) by \(n\) is denoted by \(\modulo{x}{n}\).

Let $A=\left( a_{i\,j}\right) $ be a matrix of order $n$. It is usual to associate a digraph of order \(n\) to $A$, denoted by $D\!\left(A\right)$, see \cite{brualdi2008combinatorial}. The vertices of $D\!\left(A\right) $ are labeled by the integers $\left\{0,1,...,n-1\right\} $. If \(a_{i\,j} \neq 0\), there is an arc from the vertex $i$ to the vertex $j$ of weight $a_{i\,j}$. Whether \(\sigma\) is a permutation, we just write \(D(\sigma)\) instead of \(D\left(P_{\sigma}\right)\), and we talk about digraphs and permutations associated.

In this section, we show that \(aP_{n}^{s_{1}}+bP_{n}^{s_{2}}\) is associated with a digraph who has
\((n,s_{2}-s_{1})\) main cycles of length \(n\backslash(s_{2}-s_{1})\) each, 
and we give two permutations that allow to untangle this digraphs, in a sense that will be clear later.

The first untangle is given  by the permutation associated with the matrix \(P_{n}^{s_{1}}\), because \(aP_{n}^{s_{1}}+bP_{n}^{s_{2}}= P_{n}^{s_{1}}\left( aI_{n}+bP_{n}^{s_{2}-s_{1}}\right)\). Therefore, we just need to study only digraphs associated to matrices of the form \(aI_{n}+bP_{n}^{s}\). In Figures \ref{D806_0} and \ref{D806_1} it can be seen how this permutation untangle these digraphs.

Notice that  \(D\left(aI_{n}+bP_{n}^{s} \right) \) has \((n,s)\) connected components of order \(n \backslash s\), each of them has a main spanning cycle. Notice that for  \(u,v\in [n]\), in \(D\left(aI_{n}+bP_{n}^{s} \right)\) we have a directed arc from \(u\) to \(v\) if and only if \(v-u=\modulo{s}{n}\).

%%%%%%%%%%%%%%%%%%%%%%%%%%%%%%%%%%%%%%%%%%%%%%%%%%%%%
%PRIMERA FIGURA!!!!! (NO COMFUNDIR CON LA QUE SIGUE!!!!)
\begin{figure}[h]
	\centering
	\begin{tikzpicture}[scale=0.5,commutative diagrams/every diagram]
	%Vertices y sus etiquetas y loops al arco
	\node[draw,circle,scale=0.6] (a) at (0,7.24)  {0} ;%edge [blue,loop above] node {$a$} ();
	\node[draw,circle,scale=0.6] (b) at (3,7.24) {1} ;%edge [loop above] node {$a$} ();
	\node[draw,circle,scale=0.6] (c) at (5.12,5.12) {2} ;%edge [blue,loop right] node {$a$} ();
	\node[draw,circle,scale=0.6] (d) at (5.12,2.12) {3} ;%edge [loop right] node {$a$} ();
	\node[draw,circle,scale=0.6] (e) at (3,0) {4} ;%edge [blue,loop below] node {$a$} ();
	\node[draw,circle,scale=0.6] (f) at (0,0) {5} ;%edge [loop below] node {$a$} ();
	\node[draw,circle,scale=0.6] (g) at (-2.12,2.12) {6} ;%edge [blue,loop left] node {$a$} ();
	\node[draw,circle,scale=0.6] (h) at (-2.12,5.12) {7} ;%edge [loop left] node {$a$} ();
	
	\begin{scope}[commutative diagrams/.cd, every arrow, every label]
	
	%Lados dirigidos y sus etiquetas b=7
	\draw[blue] (a) to[bend right] node[swap] (aux) {$b$} (h);
	\draw[blue] (h) to[bend right] node[swap] (aux) {$b$} (g);
	\draw[blue] (g) to[bend right] node[swap] (aux) {$b$} (f);
	\draw[blue] (f) to[bend right] node[swap] (aux) {$b$} (e);
	\draw[blue] (e) to[bend right] node[swap] (aux) {$b$} (d);
	\draw[blue] (d) to[bend right] node[swap] (aux) {$b$} (c);
	\draw[blue] (c) to[bend right] node[swap] (aux) {$b$} (b);
	\draw[blue] (b) to[bend right] node[swap] (aux) {$b$} (a);
	%arcos a=1
	\draw[red] (a) to[bend right] node[swap] (aux) {$a$} (b);
	\draw[red] (b) to[bend right] node[swap] (aux) {$a$} (c);
	\draw[red] (c) to[bend right] node[swap] (aux) {$a$} (d);
	\draw[red] (d) to[bend right] node[swap] {$a$} (e);
	\draw[red] (e) to[bend right] node[swap] {$a$} (f);
	\draw[red] (f) to[bend right] node[swap] {$a$} (g);
	\draw[red] (g) to[bend right] node[swap] {$a$} (h);
	\draw[red] (h) to[bend right] node[swap] {$a$} (a);
	
	\end{scope}
	
	%%%%%%%%%%%%%%%%%%%%%%%%%%%%%%%%%%%%%%%%%%%%
	
	%Grafo normalizado (enrulado)
	%corrimiento
	\def\movimiento{11}
	\node[draw,circle,scale=0.6] (a) at (0+\movimiento,7.24)  {0} edge [loop above,red] node {$a$} ();
	\node[draw,circle,scale=0.6] (b) at (3+\movimiento,7.24) {1} edge [loop above,red] node {$a$} ();
	\node[draw,circle,scale=0.6] (c) at (5.12+\movimiento,5.12) {2} edge [loop right,red] node {$a$} ();
	\node[draw,circle,scale=0.6] (d) at (5.12+\movimiento,2.12) {3} edge [loop right,red] node {$a$} ();
	\node[draw,circle,scale=0.6] (e) at (3+\movimiento,0) {4} edge [loop below,red] node {$a$} ();
	\node[draw,circle,scale=0.6] (f) at (0+\movimiento,0) {5} edge [loop below,red] node {$a$} ();
	\node[draw,circle,scale=0.6] (g) at (-2.12+\movimiento,2.12) {6} edge [loop left,red] node {$a$} ();
	\node[draw,circle,scale=0.6] (h) at (-2.12+\movimiento,5.12) {7} edge [loop left,red] node {$a$} ();
	
	\begin{scope}[commutative diagrams/.cd, every arrow, every label]
	%lados y sus eriquetas

	%Lados dirigidos y sus etiquetas
	\draw[blue] (c) to node (aux) {$b$} (a);
	\draw[blue] (e) to node (aux) {$b$} (c);
	\draw[blue] (g) to node (aux) {$b$} (e);
	\draw[blue] (a) to node (aux) {$b$} (g);
	\draw[blue] (d) to node (aux) {$b$} (b);
	\draw[blue] (f) to node (aux) {$b$} (d);
	\draw[blue] (h) to node (aux) {$b$} (f);
	\draw[blue] (b) to node (aux) {$b$} (h);
	\end{scope}
	\end{tikzpicture}
	\caption{On the left $D\left(aP_{8}+bP_{8}^{7}\right)$,  and on the right $D\left(aI_{8}+bP_{8}^{6}\right)$.}\label{D806_0}
\end{figure}
%%%%%%%%%%%%%%%%%%%%%%%%%%%%%%%%%%%%%%%%%%%%%%%%%%%%%%%%%%%%%%
%%%%%%%%%%%%%

%%%%%%%%%%%%%%%%%%%%%%%%%%%%%%%%%%%%%%%%%%%%%%%%%%%%%%%%%%%%%

%SEGUNDA FIGURA 
%%%%%%%%%%%%%%%%%%%%%%%%%%%%%%%%%%%%%%%%%%%%%%%%%%%%

\begin{figure}[h]
	\centering
	\begin{tikzpicture}[scale=0.6,commutative diagrams/every diagram]
	%Vertices y sus etiquetas y loops al arco
	\node[draw,circle,scale=0.6] (a) at (1,5.67)  {0} ;%edge [blue,loop above] node {$a$} ();
	\node[draw,circle,scale=0.6] (b) at (2.88,4.99) {1} ;%edge [red,loop above] node {$a$} ();
	\node[draw,circle,scale=0.6] (c) at (3.88,3.26) {2} ;%edge [loop right] node {$a$} ();
	\node[draw,circle,scale=0.6] (d) at (3.53,1.29) {3} ;%edge [blue,loop right] node {$a$} ();
	\node[draw,circle,scale=0.6] (e) at (2,0) {4} ;%edge [red,loop below] node {$a$} ();
	\node[draw,circle,scale=0.6] (f) at (0,0) {5} ;%edge [loop below] node {$a$} ();
	\node[draw,circle,scale=0.6] (g) at (-1.53,1.29) {6} ;%edge [blue,loop left] node {$a$} ();
	\node[draw,circle,scale=0.6] (h) at (-1.88,3.26) {7} ;%edge [red,loop left] node {$a$} ();
	\node[draw,circle,scale=0.6] (l) at (-0.88,4.99) {8} ;%edge [loop above] node {$a$} ();
	
	\begin{scope}[commutative diagrams/.cd, every arrow, every label]
	
	%Lados dirigidos y sus etiquetas b=5
	\draw[blue] (a) to node (aux) {} (f) ;
	\draw[blue] (b) to node (aux) {} (g);
	\draw[blue] (c) to node (aux) {} (h);
	\draw[blue] (d) to node (aux) {} (l);
	\draw[blue] (e) to node (aux) {} (a);
	\draw[blue] (f) to node (aux) {} (b);
	\draw[blue] (g) to node (aux) {} (c);
	\draw[blue] (h) to node (aux) {} (d);
	\draw[blue] (l) to node (aux) {} (e);
%	\draw[blue] (l) to node (aux) {$b$} (f);
	%arcos a=2
	\draw[red] (a) to node[swap] (aux) {} (c);
	\draw[red] (b) to node[swap] (aux) {} (d);
	\draw[red] (c) to node[swap] (aux) {} (e);
	\draw[red] (d) to node[swap] {} (f);
	\draw[red] (e) to node[swap] {} (g);
	\draw[red] (f) to node[swap] {} (h);
	\draw[red] (g) to node[swap] {} (l);
	\draw[red] (h) to node[swap] {} (a);
	\draw[red] (l) to node[swap] {} (b);
	\end{scope}
	
	%%%%%%%%%%%%%%%%%%%%%%%%%%%%%%%%%%%%%%%%%%%%
	
	%Grafo normalizado (enrulado)
	%corrimiento
	\def\movimiento{10}
	\node[draw,circle,scale=0.6] (f) at (1+\movimiento,5.67)  {0} edge [loop above,red] node {$a$} ();
	\node[draw,circle,scale=0.6] (e) at (2.88+\movimiento,4.99) {1} edge [loop above,red] node {$a$} ();
	\node[draw,circle,scale=0.6] (d) at (3.88+\movimiento,3.26) {2} edge [loop right,red] node {$a$} ();
	\node[draw,circle,scale=0.6] (c) at (3.53+\movimiento,1.29) {3} edge [loop right,red] node {$a$} ();
	\node[draw,circle,scale=0.6] (b) at (2+\movimiento,0) {4} edge [loop below,red] node {$a$} ();
	\node[draw,circle,scale=0.6] (a) at (0+\movimiento,0) {5} edge [loop below,red] node {$a$} ();
	\node[draw,circle,scale=0.6] (l) at (-1.53+\movimiento,1.29) {6} edge [loop left,red] node {$a$} ();
	\node[draw,circle,scale=0.6] (h) at (-1.88+\movimiento,3.26) {7} edge [loop left,red] node {$a$} ();
	\node[draw,circle,scale=0.6] (g) at (-0.88+\movimiento,4.99) {8} edge [loop above,red] node {$a$} ();
	
	\begin{scope}[commutative diagrams/.cd, every arrow, every label]
	%lados y sus eriquetas
	
	\draw[blue] (d) to node[swap] (aux) {$b$} (a);
	\draw[blue] (e) to node[swap] (aux) {$b$} (b);
	\draw[blue] (f) to node[swap] (aux) {$b$} (c);
	\draw[blue] (g) to node[swap] (aux) {$b$} (d);
	\draw[blue] (h) to node[swap] (aux) {$b$} (e);
	\draw[blue] (l) to node[swap] (aux) {$b$} (f);
	\draw[blue] (a) to node[swap] (aux) {$b$} (g);
	\draw[blue] (b) to node[swap] (aux) {$b$} (h);
	\draw[blue] (c) to node[swap] (aux) {$b$} (l);
	\end{scope}

\end{tikzpicture}
\caption{The digraph at left is $D\left(aP_{9}^{2}+bP_{9}^{5}\right)$ ( \(b\)--arcs in blue and  \(a\)--arcs in red),  and the digraph at right is $D\left(aI_{9}+bP_{9}^{3}\right)$.}\label{D806_1}
\end{figure}

Given \(n\) and \(s\) positive integers such that \(0<s<n\), the \((n,s)\)--canonical permutation is
\begin{equation}\label{nu}
\nu_{n,s}:=\prod_{i=1}^{(n,s)}\nu_{n,s,i},
\end{equation}
where
\[
\nu_{n,s,i} \! := \! ( i \, (n\backslash s) \! - \! 1 \;\;\; i \, (n\backslash s) \! - \! 2 \; \cdots \; i \, (n\backslash s) \! - \! ((n\backslash s) \! - \! 1) \;\;\; i \, (n\backslash s) \! - \! (n\backslash s) ).
\]
The permutation \(\nu_{n,s}\) has \((n,s)\) cycles of length \(n\backslash s\) in the natural-cyclic-order. For instance, if we consider \(n=8\) and \(s=6\), we have that
\(\nu_{8,6}=(3\;\;2\;\;1\;\;0)(7\;\;6\;\;5\;\;4)\). Notice that 
\begin{equation}
P_{\nu_{n,s}}=I_{(n,s)}\otimes P_{n\backslash s},
\end{equation}
where \(\otimes\) denote the usual Kronecker product between matrices, see \cite{steeb2011matrix}.

Permutations \(\tau_{n}^s\) and \(\nu_{n,s}\) have the same cycle type. Therefore, for each pair \(n,s\) with \(0<s<n\), there exists a permutation \(\sigma\) such that \(\sigma \tau_{n}^{s} \sigma^{-1} = \nu_{n,s}\). The digraph \(D(\nu_{8,6})\) can be seen in Figure \ref{D806_B}. 
Notice that \(D(\nu_{8,6})\) is an untangling version of \(D\left(aI_{8}+bP_{8}^{6}\right)\), up to loops. Untangling version of digraphs of form  \(D\left(aP_{n}^{s_{1}}+bP_{n}^{s_{2}}\right)\) are digraphs of form \(D\left(aI_{n}+bP_{\nu_{n,s_{2}-s_{1}}}\right)\), see Figure \ref{D806_3} and \ref{D806_4}.

Let \(n,s_{1}\) and \(s_{2}\) be integers such that \(0 < s_{1}<s_{2}<n\). If \(n\backslash s_{1}=n\backslash s_{2}\), then \(\nu_{n,s_{1}}=\nu_{n,s_{2}}\), thus, \(D\left(\nu_{n,s_{1}}\right) = D\left(\nu_{n,s_{2}}\right)\).

The permutations \(\tau_{n}^{s_{1}}\) and \(\sigma\) give us a block diagonal form of  \(aP_{n}^{s_{1}}+bP_{n}^{s_{2}}\).
\begin{align*}
P_{\sigma} P_{9}^{7} \left( aP_{9}^{2}+bP_{9}^{5} \right) P_{\sigma}^{T} &= aI_{9}+bP_{\nu_{9,3}} \\
&= \left[
\begin{array}{rrr|rrr|rrr}
	a & b & 0 & 0 & 0 & 0 & 0 & 0 & 0 \\
	0 & a & b & 0 & 0 & 0 & 0 & 0 & 0 \\
	b & 0 & a & 0 & 0 & 0 & 0 & 0 & 0 \\
	\hline
	0 & 0 & 0 & a & b & 0 & 0 & 0 & 0 \\
	0 & 0 & 0 & 0 & a & b & 0 & 0 & 0 \\
	0 & 0 & 0 & b & 0 & a & 0 & 0 & 0 \\
	\hline
	0 & 0 & 0 & 0 & 0 & 0 & a & b & 0 \\
	0 & 0 & 0 & 0 & 0 & 0 & 0 & a & b \\
	0 & 0 & 0 & 0 & 0 & 0 & b & 0 & a
\end{array}
\right] \\
&= I_{3} \otimes \left(aI_{3} + bP_{3} \right).
\end{align*}

%%%%%%%%%%%%%%%%%%%%%%%%%%%%%%%%%%%%%%%%%%%%%%%%%%%%%
%
%%%%%%%%%%%%%%%%%%%%%%%%%%%%%%%%%%%%%%%%%%%%%%%%%%%%%
\begin{figure}[h]
	\centering
	\begin{tikzpicture}[scale=0.5,commutative diagrams/every diagram]
	%Vertices y sus etiquetas y loops al arco
	\node[white] at (13,3.62) {\textcolor{black}{
			\begin{tabular}{ccc}
			$i$ & $\modulo{\varrho_{8,6}(i)}{(8,6)}$ 	& $\ell_{8,6}(i)$   \\
			0   & 0						& 0					\\
			1	& 0						& 3					\\
			2	& 0						& 2					\\
			3	& 0						& 1					\\
			4	& 1						& 0					\\
			5	& 1						& 3					\\
			6	& 1						& 2					\\
			7	& 1						& 1					
			\end{tabular}
		}
	};
	\node[draw,circle,scale=0.6] (a) at (0,7.24)  {0};% edge [blue, loop above] node {} ();
	\node[draw,circle,scale=0.6] (b) at (3,7.24) {1};% edge [blue,loop above] node {$a$} ();
	\node[draw,circle,scale=0.6] (c) at (5.12,5.12) {2};% edge [blue,loop right] node {$a$} ();
	\node[draw,circle,scale=0.6] (d) at (5.12,2.12) {3};% edge [blue,loop right] node {$a$} ();
	\node[draw,circle,scale=0.6] (e) at (3,0) {4};% edge [loop below] node {$a$} ();
	\node[draw,circle,scale=0.6] (f) at (0,0) {5};% edge [loop below] node {$a$} ();
	\node[draw,circle,scale=0.6] (g) at (-2.12,2.12) {6};% edge [loop left] node {$a$} ();
	\node[draw,circle,scale=0.6] (h) at (-2.12,5.12) {7};% edge [loop left] node {$a$} ();
	
	\begin{scope}[commutative diagrams/.cd, every arrow, every label]
	%Lados dirigidos y sus etiquetas
	\draw[blue] (a) to node (aux) {} (b);
	\draw[blue] (b) to node (aux) {} (c);
	\draw[blue] (c) to node (aux) {} (d);
	\draw[blue] (d) to node (aux) {} (a);
	\draw[blue] (e) to node (aux) {} (f);
	\draw[blue] (f) to node (aux) {} (g);
	\draw[blue] (g) to node (aux) {} (h);
	\draw[blue] (h) to node (aux) {} (e);
	\end{scope}
	\end{tikzpicture}
	\caption{$D(\nu_{8,6})=D(\nu_{8,2})$. Note that $(8,6)=(8,2)=2$ and $8\backslash 6=8\backslash 2=4$.}\label{D806_B}
\end{figure}
%%%%%%%%%%%%%%%%%%%%%%%%%%%%%%%%%%%%%%%%%%%%%%%%%%%%%%%%%%%%%%
%%%%%%%%%%%%%

%%%%%%%%%%%%%%%%%%%%%%%%%%%%%%%%%%%%%%%%%%%%%%%%%%%%%
%PRIMERA FIGURA!!!!! (NO COMFUNDIR CON LA QUE SIGUE!!!!)
\begin{figure}[h]
	\centering
	\begin{tikzpicture}[scale=0.5,commutative diagrams/every diagram]
	%Vertices y sus etiquetas y loops al arco
	\node[draw,circle,scale=0.6] (a) at (0,7.24)  {0} ;%edge [blue,loop above] node {$a$} ();
	\node[draw,circle,scale=0.6] (b) at (3,7.24) {1} ;%edge [loop above] node {$a$} ();
	\node[draw,circle,scale=0.6] (c) at (5.12,5.12) {2} ;%edge [blue,loop right] node {$a$} ();
	\node[draw,circle,scale=0.6] (d) at (5.12,2.12) {3} ;%edge [loop right] node {$a$} ();
	\node[draw,circle,scale=0.6] (e) at (3,0) {4} ;%edge [blue,loop below] node {$a$} ();
	\node[draw,circle,scale=0.6] (f) at (0,0) {5} ;%edge [loop below] node {$a$} ();
	\node[draw,circle,scale=0.6] (g) at (-2.12,2.12) {6} ;%edge [blue,loop left] node {$a$} ();
	\node[draw,circle,scale=0.6] (h) at (-2.12,5.12) {7} ;%edge [loop left] node {$a$} ();
	
	\begin{scope}[commutative diagrams/.cd, every arrow, every label]
	
	%Lados dirigidos y sus etiquetas b=7
	\draw[blue] (a) to[bend right] node[swap] (aux) {$b$} (h);
	\draw[blue] (h) to[bend right] node[swap] (aux) {$b$} (g);
	\draw[blue] (g) to[bend right] node[swap] (aux) {$b$} (f);
	\draw[blue] (f) to[bend right] node[swap] (aux) {$b$} (e);
	\draw[blue] (e) to[bend right] node[swap] (aux) {$b$} (d);
	\draw[blue] (d) to[bend right] node[swap] (aux) {$b$} (c);
	\draw[blue] (c) to[bend right] node[swap] (aux) {$b$} (b);
	\draw[blue] (b) to[bend right] node[swap] (aux) {$b$} (a);
	%arcos a=1
	\draw[red] (a) to[bend right] node[swap] (aux) {$a$} (b);
	\draw[red] (b) to[bend right] node[swap] (aux) {$a$} (c);
	\draw[red] (c) to[bend right] node[swap] (aux) {$a$} (d);
	\draw[red] (d) to[bend right] node[swap] {$a$} (e);
	\draw[red] (e) to[bend right] node[swap] {$a$} (f);
	\draw[red] (f) to[bend right] node[swap] {$a$} (g);
	\draw[red] (g) to[bend right] node[swap] {$a$} (h);
	\draw[red] (h) to[bend right] node[swap] {$a$} (a);
	
	\end{scope}
	
	%%%%%%%%%%%%%%%%%%%%%%%%%%%%%%%%%%%%%%%%%%%%
	
	%Grafo normalizado (enrulado)
	%corrimiento
	\def\movimiento{11}
	\node[draw,circle,scale=0.6] (a) at (0+\movimiento,7.24)  {0}  edge [blue, loop above,red] node {} ();
	\node[draw,circle,scale=0.6] (b) at (3+\movimiento,7.24) {1} edge [blue,loop above,red] node {$a$} ();
	\node[draw,circle,scale=0.6] (c) at (5.12+\movimiento,5.12) {2}  edge [blue,loop right,red] node {$a$} ();
	\node[draw,circle,scale=0.6] (d) at (5.12+\movimiento,2.12) {3} edge [blue,loop right,red] node {$a$} ();
	\node[draw,circle,scale=0.6] (e) at (3+\movimiento,0) {4}  edge [loop below,red] node {$a$} ();
	\node[draw,circle,scale=0.6] (f) at (0+\movimiento,0) {5}  edge [loop below,red] node {$a$} ();
	\node[draw,circle,scale=0.6] (g) at (-2.12+\movimiento,2.12) {6}  edge [loop left,red] node {$a$} ();
	\node[draw,circle,scale=0.6] (h) at (-2.12+\movimiento,5.12) {7}  edge [loop left,red] node {$a$} ();
	
	\begin{scope}[commutative diagrams/.cd, every arrow, every label]
	%Lados dirigidos y sus etiquetas
	\draw[blue] (a) to node (aux) {$b$} (b);
	\draw[blue] (b) to node (aux) {$b$} (c);
	\draw[blue] (c) to node (aux) {$b$} (d);
	\draw[blue] (d) to node (aux) {$b$} (a);
	\draw[blue] (e) to node (aux) {$b$} (f);
	\draw[blue] (f) to node (aux) {$b$} (g);
	\draw[blue] (g) to node (aux) {$b$} (h);
	\draw[blue] (h) to node (aux) {$b$} (e);
	\end{scope}
	\end{tikzpicture}
	\caption{On the left $D\left(aP_{8}+bP_{8}^{7}\right)$ and on the right its untangled version  $D\left(aI_{8}+bP_{\nu_{8,6}}\right)$.}\label{D806_3}
\end{figure}
%%%%%%%%%%%%%%%%%%%%%%%%%%%%%%%%%%%%%%%%%%%%%%%%%%%%%%%%%%%%%%
%%%%%%%%%%%%%

%SEGUNDA FIGURA 
%%%%%%%%%%%%%%%%%%%%%%%%%%%%%%%%%%%%%%%%%%%%%%%%%%%%

\begin{figure}[h]
	\centering
	\begin{tikzpicture}[scale=0.6,commutative diagrams/every diagram]
	%Vertices y sus etiquetas y loops al arco
	\node[draw,circle,scale=0.6] (a) at (1,5.67)  {0} ;%edge [blue,loop above] node {$a$} ();
	\node[draw,circle,scale=0.6] (b) at (2.88,4.99) {1} ;%edge [red,loop above] node {$a$} ();
	\node[draw,circle,scale=0.6] (c) at (3.88,3.26) {2} ;%edge [loop right] node {$a$} ();
	\node[draw,circle,scale=0.6] (d) at (3.53,1.29) {3} ;%edge [blue,loop right] node {$a$} ();
	\node[draw,circle,scale=0.6] (e) at (2,0) {4} ;%edge [red,loop below] node {$a$} ();
	\node[draw,circle,scale=0.6] (f) at (0,0) {5} ;%edge [loop below] node {$a$} ();
	\node[draw,circle,scale=0.6] (g) at (-1.53,1.29) {6} ;%edge [blue,loop left] node {$a$} ();
	\node[draw,circle,scale=0.6] (h) at (-1.88,3.26) {7} ;%edge [red,loop left] node {$a$} ();
	\node[draw,circle,scale=0.6] (l) at (-0.88,4.99) {8} ;%edge [loop above] node {$a$} ();
	
	\begin{scope}[commutative diagrams/.cd, every arrow, every label]
	
	%Lados dirigidos y sus etiquetas b=5
	\draw[blue] (a) to node (aux) {} (f) ;
	\draw[blue] (b) to node (aux) {} (g);
	\draw[blue] (c) to node (aux) {} (h);
	\draw[blue] (d) to node (aux) {} (l);
	\draw[blue] (e) to node (aux) {} (a);
	\draw[blue] (f) to node (aux) {} (b);
	\draw[blue] (g) to node (aux) {} (c);
	\draw[blue] (h) to node (aux) {} (d);
	\draw[blue] (l) to node (aux) {} (e);
	%	\draw[blue] (l) to node (aux) {$b$} (f);
	%arcos a=2
	\draw[red] (a) to node[swap] (aux) {} (c);
	\draw[red] (b) to node[swap] (aux) {} (d);
	\draw[red] (c) to node[swap] (aux) {} (e);
	\draw[red] (d) to node[swap] {} (f);
	\draw[red] (e) to node[swap] {} (g);
	\draw[red] (f) to node[swap] {} (h);
	\draw[red] (g) to node[swap] {} (l);
	\draw[red] (h) to node[swap] {} (a);
	\draw[red] (l) to node[swap] {} (b);
	\end{scope}
	
	%%%%%%%%%%%%%%%%%%%%%%%%%%%%%%%%%%%%%%%%%%%%
	
	%Grafo normalizado (enrulado)
    %corrimiento
    \def\movimiento{10}
    \node[draw,circle,scale=0.6] (f) at (1+\movimiento,5.67)  {0} edge [loop above,red] node {$a$} ();
    \node[draw,circle,scale=0.6] (e) at (2.88+\movimiento,4.99) {1} edge [loop above,red] node {$a$} ();
    \node[draw,circle,scale=0.6] (d) at (3.88+\movimiento,3.26) {2} edge [loop right,red] node {$a$} ();
    
    \node[draw,circle,scale=0.6] (c) at (3.53+\movimiento,1.29) {3} edge [black,loop right,red] node {$a$} ();
    \node[draw,circle,scale=0.6] (b) at (2+\movimiento,0) {4} edge [black,loop below,red] node {$a$} ();
    \node[draw,circle,scale=0.6] (a) at (0+\movimiento,0) {5} edge [black,loop below,red] node {$a$} ();
    \node[draw,circle,scale=0.6] (l) at (-1.53+\movimiento,1.29) {6} edge [loop left,red] node {$a$} ();
    \node[draw,circle,scale=0.6] (h) at (-1.88+\movimiento,3.26) {7} edge [loop left,red] node {$a$} ();
    \node[draw,circle,scale=0.6] (g) at (-0.88+\movimiento,4.99) {8} edge [loop above,red] node {$a$} ();
    
    \begin{scope}[commutative diagrams/.cd, every arrow, every label]
    %lados y sus eriquetas
    %a=5
    %d=2
    %e=1
    %b=4
    %f=0
    %c=3
    %g=8
    %h=7
    %l=6
    \draw[blue] (e) to node (aux) {$b$} (d);
    \draw[blue] (f) to node (aux) {$b$} (e);
    \draw[blue] (d) to node (aux) {$b$} (f);
    
    \draw[blue] (b) to node (aux) {$b$} (a);
    \draw[blue] (a) to node (aux) {$b$} (c);
    \draw[blue] (c) to node (aux) {$b$} (b);
    
    \draw[blue] (h) to node (aux) {$b$} (g);
    \draw[blue] (l) to node (aux) {$b$} (h);
    \draw[blue] (g) to node (aux) {$b$} (l);
    \end{scope}
	
	\end{tikzpicture}
	\caption{The Digraph at left is $D\left(aP_{9}^{2}+bP_{9}^{5}\right)$ (in blue \(b\)--arcs and in red \(a\)--arcs),  and the digraph at right is its untangled version  $D\left(aI_{9}+bP_{\nu_{9,3}}\right)$.}\label{D806_4}
\end{figure}
%%%%%%%%%%%%%%%%%%%%%%%%%%%%%%%%%%%%%%%%%%%%%%%%%%%%%%%%%%%%%%
%%%%%%%%%%%%%

%%%%%%%%%%%%%%%%%%%%%%%%%%%%%%%%%%%%%%%%%%%%%%%%%%%
%
%%%
%
%%%%%%%%%%%%%%%%%%%%%%%%%%%%%%%%%%%%%%%%%%%%%%%%%%%
\section{Finding \(\sigma_{n,s}\)}\label{Sec_Sigma}

In order to compute the Drazin (group) inverse of a matrix of the form \(aP_{n}^{s_{1}}+bP_{n}^{s_{2}}\) (among others computations) we need to find one \(\sigma\) explicitly. This can be done if we find which vertices are together in the same connected component of \(D(aI_{n}+bP_{n}^{s})\). Since this digraph appears many times in our work, we just write \(D_{n,s}(a,b)\) instead of \(D\left(aI_{n}+bP_{n}^{s}\right)\).

Let $n$ and $s$ be non-negative integers such that $0< s <n$. For each \(i \in \mathbb{Z}\), we define 
%\begin{equation}
%R(n,s,i):= \{i+ks : k \in \mathbb{Z}\}.
%\end{equation}
%and 
\begin{equation}
R(n,s,i):= \{\modulo{i+k\,s}{n} : k \in \mathbb{Z}\}.
\end{equation}
This is the set of all reachable vertices of from \(i\) in \(D_{n,s}(a,b)\). Notice that the following statements are all equivalent.
\begin{enumerate}
	\item \(R(n,s,i_{1}) =R(n,s,i_{2}) \),
	\item \(\modulo{i_2}{n} \in R(n,s,i_{1})\),
	\item \(R(n,s,i_{1}) \cap R(n,s,i_{2})\neq \emptyset\), and
	\item \(i_{1}-i_{2}=\modulo{0}{(n,s)}\).
\end{enumerate}
In Figure \ref{D806}, we can see \(R(8,6,0)\) and \(R(8,6,1)\).
%For example, w.l.g. assume \(i_{1}>i_{2}\)  \(i if \(\overline{R(n,s,i_{1})} =\overline{R(n,s,i_{2})} \), then \( i_{1}=\modulo{ks+i_{2}}\). Therefore, \(i_{1}-i_{2}= \modulo{0}{(n,s)}\),
%%%%%%%%%%%%%%%%%%%%%%%%%%%%%%%%%%%%%%%%%%%%%%%%%%%%%
%
%%%%%%%%%%%%%%%%%%%%%%%%%%%%%%%%%%%%%%%%%%%%%%%%%%%%%
\begin{figure}[h]
	\centering
	\begin{tikzpicture}[scale=0.5,commutative diagrams/every diagram]
	%Vertices y sus etiquetas y loops al arco
	\node[white] at (13,3.62) {\textcolor{black}{$
			\begin{matrix}
			\textcolor{blue}{R(8,6,0)}&\textcolor{blue}{ =} & \textcolor{blue}{	\{0,6,4,2\} }\\
			{} & {} & {} \\
			%			\textcolor{blue}{R(8,6,0)} &\textcolor{blue}{ =}
			%			\textcolor{blue}{R(8,6,2) = R(8,6,4) = R(8,6,6) },\\
			%			{}& {} & {} \\
			R(8,6,1) & = & \{1,7,5,3\} \\
			\end{matrix}$
	}};
	\node[draw,blue,circle,scale=0.6] (a) at (0,7.24)  {0} edge [blue,loop above] node {$a$} ();
	\node[draw,circle,scale=0.6] (b) at (3,7.24) {1} edge [loop above] node {$a$} ();
	\node[draw,blue,circle,scale=0.6] (c) at (5.12,5.12) {2} edge [blue,loop right] node {$a$} ();
	\node[draw,circle,scale=0.6] (d) at (5.12,2.12) {3} edge [loop right] node {$a$} ();
	\node[draw,blue,circle,scale=0.6] (e) at (3,0) {4} edge [blue,loop below] node {$a$} ();
	\node[draw,circle,scale=0.6] (f) at (0,0) {5} edge [loop below] node {$a$} ();
	\node[draw,blue,circle,scale=0.6] (g) at (-2.12,2.12) {6} edge [blue,loop left] node {$a$} ();
	\node[draw,circle,scale=0.6] (h) at (-2.12,5.12) {7} edge [loop left] node {$a$} ();
	
	\begin{scope}[commutative diagrams/.cd, every arrow, every label]
	%Lados dirigidos y sus etiquetas
	\draw[blue] (c) to node (aux) {$b$} (a);
	\draw[blue] (e) to node (aux) {$b$} (c);
	\draw[blue] (g) to node (aux) {$b$} (e);
	\draw[blue] (a) to node (aux) {$b$} (g);
	\draw (d) to node (aux) {$b$} (b);
	\draw (f) to node (aux) {$b$} (d);
	\draw (h) to node (aux) {$b$} (f);
	\draw (b) to node (aux) {$b$} (h);
	\end{scope}
	\end{tikzpicture}
	\caption{$D_{8,6}(a,b)$. Note that $(8,6)=2$ and $8\backslash 6=4$.}\label{D806}
\end{figure}
%%%%%%%%%%%%%%%%%%%%%%%%%%%%%%%%%%%%%%%%%%%%%%%%%%%%%%%%%%%%%%
%%%%%%%%%%%%%

%%%%%%%%%%%%%%%%%%%%%%%%%%%%%%%%%%%%%%%%%%%%%%%%%%%%%%%%%%%%%
%
%%%%%%%%%%%%%%%%%%%%%%%%%%%%%%%%%%%%%%%%%%%%%%%%%%%%%
\begin{figure}[h]
	\centering
	\begin{tikzpicture}[scale=0.7,commutative diagrams/every diagram]
	%Vertices y sus etiquetas y loops al arco
	\node[white] at (10,3.62) {\textcolor{black}{
			\begin{tabular}{ccc}
			$i$ & $\cycle_{8,6}(i)$ 	& $\pos_{8,6}(i)$\\
			\textcolor{blue}{0}  & \textcolor{blue}{0}					& \textcolor{blue}{0}					\\
			\textcolor{red}{1}	& \textcolor{red}{1}						& \textcolor{red}{0}					\\
			2	& 2						& 0 				\\
			\textcolor{blue}{3}	&\textcolor{blue}{0}					& \textcolor{blue}{1}					\\
			\textcolor{red}{4}	& \textcolor{red}{1}						& \textcolor{red}{1}					\\
			5	& 2						& 1					\\
			\textcolor{blue}{6}	& \textcolor{blue}{0}						& \textcolor{blue}{2}				\\
			\textcolor{red}{7}	& \textcolor{red}{1}						& \textcolor{red}{2}				\\
			8	& 2						& 2					
			\end{tabular}
		}
	};
	\node[draw,blue,circle,scale=0.6] (f) at (1,5.67)  {0} edge [blue,loop above] node {$a$} ();
	\node[draw,red,circle,scale=0.6] (e) at (2.88,4.99) {1} edge [red,loop above] node {$a$} ();
	\node[draw,circle,scale=0.6] (d) at (3.88,3.26) {2} edge [loop right] node {$a$} ();
	\node[draw,blue,circle,scale=0.6] (c) at (3.53,1.29) {3} edge [blue,loop right] node {$a$} ();
	\node[draw,red,circle,scale=0.6] (b) at (2,0) {4} edge [red,loop below] node {$a$} ();
	\node[draw,circle,scale=0.6] (a) at (0,0) {5} edge [loop below] node {$a$} ();
	\node[draw,blue,circle,scale=0.6] (l) at (-1.53,1.29) {6} edge [blue,loop left] node {$a$} ();
	\node[draw,red,circle,scale=0.6] (h) at (-1.88,3.26) {7} edge [red,loop left] node {$a$} ();
	\node[draw,circle,scale=0.6] (g) at (-0.88,4.99) {8} edge [loop above] node {$a$} ();
	
	\begin{scope}[commutative diagrams/.cd, every arrow, every label]
	%Lados dirigidos y sus etiquetas
	\draw[blue] (f) to node[swap] (aux) {$b$} (c);
	\draw[blue] (c) to node[swap] (aux) {$b$} (l);
	\draw[blue] (l) to node[swap] (aux) {$b$} (f);
	\draw[red] (e) to node[swap] (aux) {$b$} (b);
	\draw[red] (b) to node[swap] (aux) {$b$} (h);
	\draw[red] (h) to node[swap] (aux) {$b$} (e);
	\draw (d) to node[swap] (aux) {$b$} (a);
	\draw (a) to node[swap] (aux) {$b$} (g);
	\draw (g) to node[swap] (aux) {$b$} (d);
	\end{scope}
	\end{tikzpicture}
	\caption{$D_{9,3}(a,b)$. Note that $(9,3)=9\backslash 3=3$.}\label{D903}
\end{figure}
%%%%%%%%%%%%%%%%%%%%%%%%%%%%%%%%%%%%%%%%%%%%%%%%%%%%%%%%%%%%%%
In order to obtain  explicitly a \(\sigma\), we introduce the following functions:
\begin{enumerate}
	\item \(\varrho_{n,s}: [n] \longrightarrow [(n,s)]\),
	\[\varrho_{n,s}(i):=\max \{ k \in [(n,s)] : i\geq k\, (n\backslash s) \}.\]
	\item \(\ell_{n,s} : [n] \longrightarrow [n\backslash s]\),
	\[\ell_{n,s}(i):=\modulo{i -  \varrho_{n,s}(i) \, (n\backslash s)}{n\backslash s}.\]
	\item \(\cycle_{n,s}: [n] \longrightarrow [(n,s)]\), if \(i\in R(n,s,h)\), then
	\[\cycle_{n,s}(i):= h.\]
	\item \(\pos_{n,s}: [n] \longrightarrow [n\backslash s]\), if \(i\in R(n,s,h)\) and \(i=t\,s+h\), then
	\[\pos_{n,s}(i)= t.\]
\end{enumerate}

The functions \(\varrho_{n,s}\), and \(\ell_{n,s}\) give us, essentially,  the cycle and the position on the cycle of a vertex of \(D(\nu_{n,s})\), respectively. While the functions \(\cycle_{n,s}\) and \(\pos_{n,s}\) give us the main cycle and the position on the main cycle in \(D_{n,s}(a,b)\), respectively.

We embed these digraphs in the cylinder \([(n,s)] \times [n\backslash s]\). The following functions and their pullbacks are these embeddings.
\begin{enumerate}
	\setcounter{enumi}{4}
	\item \(J_{n,s}: [n] \longrightarrow [(n,s)] \times [n \backslash s]\)
	\[J_{n,s}(i)=(\varrho_{n,s}(i),\ell_{n,s}(i)).\]
	\item\(\widehat{J}_{n,s}: [(n,s)] \times [n \backslash s] \longrightarrow [n]\)
	\[\widehat{J}_{n,s}(x,y)= y+ x \, (n\backslash s).\]
	\item \(F_{n,s}: [n] \longrightarrow  [(n,s)] \times [n \backslash s]\)
	\[F_{n,s}(i)=(\cycle_{n,s}(i),\pos_{n,s}(i)).\]
	\item \(\widehat{F}_{n,s}: [(n,s)] \times [n \backslash s] \longrightarrow [n]\)
	\[\hat{F}_{n,s}(x,y)= \modulo{s\,y + x}{n}.\]
\end{enumerate}
The function \(J_{n,s}\) embeds \(D(\nu_{n,s})\) in \( [(n,s)] \times [n \backslash s]\) ``in the same way'' that \(F_{n,s}\) embeds  \(D_{n,s}(a,b)\) in \( [(n,s)] \times [n \backslash s]\).

Let \(A\) be a set, we denote the identity map by \(id_{A}:A \rightarrow A\). If \(n\) and \(s\) are two integers such that \(0<s<n\), then the following identities are direct
\begin{align*}
\widehat{J}_{n,s} \circ J_{n,s} &=id_{[n]},\\
J_{n,s} \circ \widehat{J}_{n,s} &= id_{[(n,s)] \times [n \backslash s]},\\
\widehat{F}_{n,s} \circ F_{n,s} & =id_{[n]},\\
F_{n,s} \circ \widehat{F}_{n,s} & = id_{[(n,s)] \times [n \backslash s]}.
\end{align*}
These expressions give us the next lemma.
\begin{lemma}
	Let \(n\) and \(s\) be integers such that \(0<s<n\). Then, the function \(\sigma_{n,s}: [n] \longrightarrow [n]\), defined by
	\(\sigma_{n,s}:= \widehat{F}_{n,s} \circ J_{n,s}\), is a permutation of \([n]\) and  \(\sigma_{n,s}^{-1} = \widehat{J}_{n,s} \circ F_{n,s}\).
\end{lemma}
  
The function \(\Shift_{n,s}: [(n,s)] \times [n \backslash s] \longrightarrow [(n,s)] \times [n \backslash s]\), defined by
\[
\Shift_{n,s}(x,y):=(x,\,\modulo{y-1}{n\backslash s}),
\]
allows us to express \(\tau_{n}^{s}\) in terms of \(F_{n,s}\) and \(\widehat{F}_{n,s}\), and \(\nu_{n,s}\) in terms of \(J_{n,s}\) and \(\widehat{J}_{n,s}\).

\begin{lemma}
Let \(n\) and \(s\) be integers such that \(0<s<n\). Then
\begin{align*}
\tau_{n}^{s}&=  \widehat{F}_{n,s} \circ \Shift_{n,s} \circ F_{n,s},\\
\nu_{n,s} &= \widehat{J}_{n,s} \circ \Shift_{n,s} \circ J_{n,s}.
\end{align*}
\end{lemma}

\begin{corollary} \label{coro_J1}
Let \(n\) and \(s\) be integers such that \(0<s<n\). Then
\begin{align*}
\nu_{n,s} & =\sigma_{n,s}^{-1} \circ \tau_{n,s} \circ \sigma_{n,s},\\
I_{(n,s)} \otimes P_{n \backslash s} & =P_{\sigma_{n,s}}^{T}P_{n}^{s}P_{\sigma_{n,s}}.%\label{Ecuacion_Sigma_s}
\end{align*}
\end{corollary}

\begin{theorem} \label{Untagling_theo}
	Let \(n, s_{1},s_{2}\) be integers such that \(0\leq s_{1}<s_{2} <n\), and let \(a,b\) non-zero complex numbers. Then, 
	\begin{equation*}
	aP_{n}^{s_{1}}+bP_{n}^{s_{2}}= P_{n}^{s_{1}} P_{\sigma_{n,s_{2}-s_{1}}}
	\left[
	I_{(n,s_{2}-s_{1})} \otimes \left(aI_{n \backslash \left(s_{2}-s_{1}\right) } +bP_{n \backslash \left(s_{2}-s_{1}\right) }\right)
	\right]
	P_{\sigma_{n,s_{2}-s_{1}}}^{T}.
	\end{equation*}
\end{theorem}
\begin{proof} %\textit{of Theorem \ref{Untagling_theo}: }
Let \(s=s_{2}-s_{1}\). Then, by Corollary \ref{coro_J1},
\begin{eqnarray*}
	P_{\sigma_{n,s}}^{T}
	P_{n}^{n-s_{1}} \left(
	aP_{n}^{s_{1}} +bP_{n}^{s_{2}} \right) 
	P_{\sigma_{n,s}} & =&
	a I_{n}+b 	P_{\sigma_{n,s}}^{T} P_{n}^{s}	P_{\sigma_{n,s}}\\
	{ } & = &  a I_{n}+b\left( I_{(n,s)} \otimes P_{n \backslash s}\right)\\
	{} & = & 	I_{(n,s)} \otimes \left(aI_{n \backslash s } +bP_{n \backslash s }\right).
\end{eqnarray*}
This conclude the proof. 
\end{proof}
%%%%%%%%%%%%%%%%%%%%%%%%%%%%%%%%%%%%%%%%%%%%%%%%%%%%%%%%%%%%%%%%%%%%%%%%%%%%%%%

%
%It is well-known that all of the elements in the symmetric group $\mathbb{S}_{n}$
%can be written as a unique product of disjoint cycles. 
%Moreover, this decomposition 
%determined the conjugance classes in $\mathbb{S}_{n}$, i.e\@
%if $\tau_1,\tau_2\in\mathbb{S}_n $ have the same cyclic structure there
%exists $\sigma\in\mathbb{S}_n$ such that $\sigma^{-1}\tau_1\sigma =\tau_2$.

The following theorem showed that we can untangle $\sum_{k=0}^{n\backslash s-1}a_{k}P_{n}^{ks}$, in the same way as in the case of $P_{n}^{s}$.

\begin{theorem}
Let $s,n$ positive integers such that $s<n$ and let $a_{k}$ be non-zero complex numbers for $k\in [n\backslash s]$. Then there exists a permutation $\sigma_{n,s}$ of \([n]\) such that
$$P_{\sigma_{n,s}}^{-1}\left( \sum_{k=0}^{(n\backslash s)-1} a_{k}P_{n}^{s\,k}\right) P_{\sigma_{n,s}}=I_{(n,s)} \otimes \left(  \sum_{k=0}^{(n\backslash s)-1} a_{k}P_{n\backslash s}^{k} \right).$$ 
\end{theorem} 
\begin{proof}
	Clearly $\nu_{n,s}$ has the same cycle type that $\tau_{n}^{s}$. Then, there exists a permutation $\sigma_{n,s}$ such that 
	$$\sigma_{n,s}^{-1}\tau_{n}^{s}\sigma_{n,s}= \nu_{n,s}.$$
	By taking into account that the application $P: \mathbb{S}_{n}\rightarrow \mathrm{GL}(n)$, where \(\mathbb{S}_{n}\) is the symmetric group and \(\mathrm{GL}(n)\)  is the general lineal group, 
	defined by $P(\sigma)=P_{\sigma}$ consider in the above section is a group homomorphism, we have that 
	$$P_{\sigma_{n,s}}^{-1}P_{n}^s P_{\sigma_{n,s}}=P_{\sigma_{n,s}^{-1}\tau_{n}^{s}\sigma_{n,s}}=P_{\nu_{n,s}}=I_{(n,s)}\otimes P_{n\backslash s}.$$  
	In the same way, if we consider the powers $P_{n}^{s\,k}$ with $k\in[n\backslash s]$, we obtain that $$P_{\sigma_{n,s}}^{-1}P_{n}^{k\,s}P_{\sigma_{n,s}}=(P_{\sigma_{n,s}}^{-1}P_{n}^{s}P_{\sigma_{n,s}})^k=P_{\sigma_{n,s}^{-1}\circ \tau_{n}^{s} \circ\sigma_{n,s}}^k=P_{\nu_{n,s}}^{k}.$$ 
	By the property $(A\otimes B)(C\otimes D)=AC\otimes BD$ of the Kronecker product, we have that 
	$$P_{\sigma_{n,s}}^{-1}P_{n}^{k\,s}P_{\sigma_{n,s}}=P_{\nu_{n,s}}^{k}=(I_{(n,s)}\otimes P_{n\backslash s})^k=I_{(n,s)}\otimes P_{n\backslash s}^{k},$$
	for all $k\in [n\backslash s]$. 
	By taking into account that $A\otimes(B+C)=A\otimes B+ A\otimes C$, we obtain
	$$P_{\sigma_{n,s}}^{-1} \! \left( \sum_{k=0}^{(n\backslash s)-1} \!\! a_{k}P_{n}^{k\,s}\right) \! P_{\sigma_{n,s}}= \!\!
	\sum_{k=0}^{(n\backslash s)-1} \!\! a_{k} \left( I_{(n,s)}\otimes P_{n\backslash s}^{k}\right)  =
	I_{(n,s)}\otimes\left(  \sum_{k=0}^{n\backslash s-1} \!\! a_{k}P_{n\backslash s}^{k}\right) ,$$
	as asserted.
\end{proof}

\medskip

%
%%%%%%
%
%%%%%%%%%%%%%%%%%%%%%%%%%%%%%%%%%%%%%%%%%%%%%%%%%%%%%%%%%%%%%%%%%%%%%%%%%%%%%%%

\section{Determinant and permanent of \(aP_{n}^{s_{1}}+bP_{n}^{s_{2}}\)}\label{Sec_Det_Perm_aP_bP}

A linear subdigraph $L$ of a digraph $D$ is a spanning subdigraph of $D$, in which each vertex has indegree 1 and outdegree 1 (there is exactly one arc get into each vertex, and ther is exactly one arc (possibly the same) get out of each vertex), see \cite{brualdi2008combinatorial}.

\begin{theorem}[\cite{brualdi2008combinatorial}] \label{Teo_a_lo_Brualdi}
	Let $A=\left( a_{ij}\right) $ be a square matrix of order $n$. Then
	\[
	\det \left( A\right) =\sum\limits_{L\,\in\, \mathcal{L}\left( D\left( A \right) \right) } (-1)^{n-c\left( L\right) }w\left( L\right),
	\]
	and
	\[
	\perm{A} = \sum\limits_{L\,\in\, \mathcal{L}\left(D(A)\right)} w(L).
	\]
	Where $ \mathcal{L}\left( D\left( A \right) \right) $ is the set of all linear subdigraphs of the digraph $D(A)$, $c\left( L\right) $ is the number of cycles contained in $L$, and $w\left( L\right) $ is the product of the weights of the edges of $L$.
\end{theorem}

\begin{theorem} \label{Teo_Determinante_1}
	Let $n$ be a non-negative integer and let $a,b$ be non-zero complex numbers. Then
	\[
	\det \left(aI_{n}+bP_{n} \right)= a^{n}-\left(-b\right)^{n}.
	\]
\end{theorem}
\begin{proof}
	Notice that the digraph \(D_{n,1}(a,b)\) has only two linear subdigraphs, the whole cycle and \(n\) loops. The result follows from Theorem \ref{Teo_a_lo_Brualdi}.
\end{proof}

\begin{corollary}
\label{Teo_Determinante}
Let $n$, $s_{1}$, and $s_{2}$ be integers such that $ 0 \leq s_{1} <s_{2} < n$. Let $a$ and  $b$ be non-zero complex numbers. Then
\[
\det \left(aP_{n}^{s_{1}}+bP_{n}^{s_{2}} \right)= (-1)^{(n-1)s_{1}} \left( a^{n \backslash \left(s_{2}-s_{1}\right)}-\left( -b\right) ^{n \backslash \left( s_{2}-s_{1}\right)}\right)^{(n,s_{2}-s_{1})}.
\]
Hence, $aP_{n}^{s_{1}}+bP_{n}^{s_{2}} $ is singular if and only if \(a^{n \backslash \left(s_{2}-s_{1}\right)}-\left( -b\right) ^{n \backslash \left( s_{2}-s_{1}\right)}=0\).
\end{corollary}
Notice that  \(a^{n \backslash \left(s_{2}-s_{1}\right)}-\left( -b\right) ^{n \backslash \left( s_{2}-s_{1}\right)}=0\) if and only if either \(a=\pm b\) when \({n \backslash \left( s_{2}-s_{1}\right)}\) is even or  \(a=-b\) when \({n \backslash \left(s_{2}-s_{1}\right)}\) is odd.

\begin{proof}
Let \(s=s_{2}-s_{1}\). By Theorem \ref{Untagling_theo}
\[
\det\left(aP_{n}^{s_{1}}+bP_{n}^{s_{2}} \right) = \det\left(P_{n}^{s_{1}} \right) \det\left(I_{(n,s)} \otimes \left(aI_{n\backslash s} + bP_{n\backslash s} \right) \right).
\]
Owing to \(\det\left(P_{n}^{i} \right) = (-1)^{(n-1)\, i}\), the result follows from Theorem \ref{Teo_Determinante_1}.
\end{proof}

\begin{example}
If we consider the matrix \(aI_{8}+bP_{8}^{6}\), 
whose associated digraph is in Figure \ref{D806}, is singular if and only if \(|a| = |b\,|\), the same occur with the matrix \(aP_{8}^{1}+bP_{8}^{7}\). 
The matrix \(aI_{9}+bP_{9}^{3}\), whose associated digraph is in Figure \ref{D903}, is singular if and only if \(a= -b\). 
The matrix \(aP_{9}^{2}+bP_{9}^{5}\) is also singular if and only if \(a = -b\).
\end{example}

\begin{corollary} \label{teo_perm2}
Let $n$, $s_{1}$, and $s_{2}$ be integers such that $ 0 \leq s_{1} <s_{2} < n$. Let $a$ and  $b$ be non-zero complex numbers. Then
\[
\perm{ \left(aP_{n}^{s_{1}}+bP_{n}^{s_{2}}\right)}=\left(a^{n\backslash \left(s_{2}-s_{1}\right)} + b^{n\backslash \left(s_{2}-s_{1}\right)}\right)^{(n,s_{2}-s_{2})}.
\]
\end{corollary}
\begin{proof}
The result follows from Theorem \ref{Teo_a_lo_Brualdi} and Corollary \ref{Teo_Determinante}: if \(A\) is an square matrix and \(Q\) is a permutation matrix of the same order, then \(\perm{QA} = \perm{A}\), see Theorem 1.1. of \cite{minc1984permanents}.
\end{proof}
%%%%%%%%%%%%%%%%%%%%%%%%%%%%%%%%%%%%%%%%%%%%%%%%%%%%%%%%%%%%%%%%%%%%%%%%%%%%%%%
%
%%%%%
\medskip
%
%%%%%%%%%%%%%%%%%%%%%%%%%%%%%%%%%%%%%%%%%%%%%%%%%%%%%%%%%%%%%%%%%%%%%%%%%%%%%%%
\section{Inverse of $aP_{n}^{s_1}+bP_{n}^{s_2}$}\label{Sec_Inv_aP_bP}

\begin{theorem} \label{Teo_inv_principal}
		Let $n$ be a non-negative integer and let $a$ and  $b$ be non-zero complex numbers. If \(aI_{n}+bP_{n}\) is non-singular, then 
	\begin{equation}
	\left(aI_{n}+bP_{n}\right)^{-1}=\dfrac{1}{a^{n}-(-b)^{n}} \sum_{i=0}^{n-1}(-1)^{i}b^{i}a^{n-1-i} P_{n}^{i}.
	\end{equation}
\end{theorem}
\begin{proof}
	It is just check that
	\[
	\left(aI_{n}+bP_{n}\right) \left[\dfrac{1}{a^{n}-(-b)^{n}} \sum_{i=0}^{n-1}(-1)^{i}b^{i}a^{n-1-i} P_{n}^{i} \right]=I_{n}.\qedhere
	\] 
\end{proof}

\begin{example}
Let $a,b$ non-zero complex numbers. If \(a^{4}-(-b)^{4} \neq 0\), then
\[\left(aI_{4}+bP_{4}\right)^{-1}=\dfrac{1}{a^4-b^4}\Circulante\left(a^3,-ba^2,b^2 a,-b^3\right).\]
If \(a^{5}-(-b)^{5} \neq 0\), then \[\left(aI_{5}+bP_{5}\right)^{-1}=\dfrac{1}{a^5+b^5}\Circulante\left(a^4,-ba^3,b^2 a^2,-b^3 a, b^4\right).\]
\end{example}

%In order to obtain an explicit formula for the inverse of a non-singular circulant matrices of the form \(aP_{n}^{s_1}+bP_{n}^{s_2}\), we define
%\begin{equation}\label{Ecuacion_Rho}
%\rho_{n,s}(i)=(-1)^{\pos_{n,s}(i)}\,\delta_{0,\cycle_{n,s}(i)} \,b^{\pos_{n,s}(i)}\,a^{n\backslash s-1-\pos_{n,s}(i)},
%\end{equation}
%where \(\delta\) is the usual Kronecker delta.
In order to obtain an explicit formula for the inverse of a non-singular circulant matrix of the form \(aP_{n}^{s_1}+bP_{n}^{s_2}\), we define
\begin{equation}\label{Ecuacion_Rho}
\rho_{n,s}(i)=(-1)^{\pos_{n,s}(i)}\,\delta_{0,\cycle_{n,s}(i)} \,b^{\pos_{n,s}(i)}\,a^{(n\backslash s)-1-\pos_{n,s}(i)},
\end{equation}
where \(\delta\) is the usual Kronecker delta.
Notice that $\rho_{n,s}$ satisfies the following properties:
    \begin{enumerate}
        \item For $n>0$, $\rho_{n,1}(i)= (-1)^{i}b^{i}a^{n-i-1}$ for all $i=0,\dots,n-1$.
        \item For $0<s<n$, $\rho_{n\backslash s,1}(i)= \rho_{n,s}(i\, s)$ for all $i=0,\dots,(n\backslash s)-1$. 
\end{enumerate}

\begin{corollary}\label{coro_inv}
Let $n$, $s_{1}$, and $s_{2}$ be integers such that $ 0 \leq s_{1} <s_{2} < n$. Let $a$ and  $b$ be non-zero complex numbers such that \(a^{n \backslash \left(s_{2}-s_{1}\right)}-\left( -b\right)^{n \backslash \left( s_{2}-s_{1}\right)}\neq 0\). Then 
\begin{equation}
\left(aP_{n}^{s_1}+bP_{n}^{s_{2}}\right)^{-1} =
\dfrac{1}{a^{n \backslash \left(s_{2}-s_{1}\right)}-(-b)^{n \backslash \left(s_{2}-s_{1}\right)}} \sum_{i=0}^{n-1} \rho_{n,s_{2}-s_{1}}(i+s_{1})P_{n}^{i}.
\end{equation}
\end{corollary}

\begin{proof}
Let \(s=s_{2}-s_{1}\). By Theorem \ref{Untagling_theo} we have that \(\left(aP_{n}^{s_{1}}+bP_{n}^{s_{2}} \right)^{-1}\) is equal to
\[
P_{\sigma_{n,s}}\left[I_{(n,s)} \otimes \left(aI_{n \backslash s } +bP_{n \backslash s}\right)^{-1} \right] P_{\sigma_{n,s}}^{T} P_{n}^{n-s_{1}}.
\]
Thus, by Theorem \ref{Teo_inv_principal} and \eqref{Ecuacion_Rho}
\[
\left(aP_{n}^{s_{1}}+bP_{n}^{s_{2}} \right)^{-1} = P_{\sigma_{n,s}} \left( \sum\limits_{i=0}^{(n\backslash s) - 1} I_{(n,s)} \otimes \frac{\rho_{n\backslash s,1}(i)}{a^{n\backslash s} - (-b)^{n\backslash s}} P_{n\backslash s}^{i} \right) P_{\sigma_{n,s}}^{T} P_{n}^{n-s_{1}}.
\]
By Corollary \ref{coro_J1} we have that 
\begin{equation}\label{Ecuacion_Potencias_P_ns1}
P_{n}^{s} = P_{\sigma_{n,s}} \left(I_{(n,s)} \otimes P_{n\backslash s} \right) P_{\sigma_{n,s}}^{T},
\end{equation}
and
\begin{equation} \label{Ecuacion_Potencias_P_ns2}
P_{\sigma_{n,s}} \left(I_{(n,s)} \otimes P_{n\backslash s}^{i} \right) P_{\sigma_{n,s}}^{T} = P_{n}^{i\, s}.
\end{equation}
Then
\begin{align*}
    \left(aP_{n}^{s_{1}}+bP_{n}^{s_{2}} \right)^{-1} &= \sum\limits_{i=0}^{(n\backslash s) - 1} \frac{\rho_{n\backslash s,1}(i)}{a^{n\backslash s} - (-b)^{n\backslash s}} P_{n}^{(i\, s) - s_{1}} \\
    &= \sum\limits_{i=0}^{(n\backslash s) - 1} \frac{\rho_{n,s}(i\, s)}{a^{n\backslash s} - (-b)^{n\backslash s}} P_{n}^{(i\, s) - s_{1}} \\
    &= \sum_{j=0}^{n-1} \frac{\rho_{n,s}(j+s_{1})}{a^{n \backslash s}-(-b)^{n \backslash s}} P_{n}^{j}. \qedhere
\end{align*}
\end{proof}

\begin{example}
Let \(a\) and \(b\) be two real numbers such \(a^{4}-(-b)^{4}\neq 0\). We will compute \(\left(aP_{12}+bP_{12}^{4}\right)^{-1}\). Since \(4-1=3\) and \(12 \backslash 3=4\), by Theorems \ref{Untagling_theo} and \ref{Teo_inv_principal}, we know that the inverse of \(aI_{12}+bP_{12}^{3}\) is composed essentially by 3 blocks of 
\[\Circulante \left(a^3,-a^2b,ab^2,-b^3\right).\] 
They are merged in a \(12\times 12\) matrix via \(P_{12}\) and \(P_{\sigma_{12,3}}\). Since \(s_{1}=1\), we have that \((12,3)=3\), so there are 3 major cycles of length 4 in \(D\left(aI_{12}+bP_{12}^{3}\right)\):

\begin{center}
\begin{tabular}{ccccccccr}
	0 & $\rightarrow$ & 3 & $\rightarrow$ & 6 & $\rightarrow$ & 9 & $\rightarrow$ & 0.\\
	1 & $\rightarrow$ & 4 & $\rightarrow$ & 7 & $\rightarrow$ & 10 & $\rightarrow$ & 1.\\
	2 & $\rightarrow$ & 5 & $\rightarrow$ & 8 & $\rightarrow$ & 11 & $\rightarrow$ & 2.\\
\end{tabular}
\end{center}
Now we compute the coefficients of the inverse
\begin{center}
	\begin{tabular} {|c|c|c|c|c|}
		\hline
		$i$ & $\cycle_{12,3}(i+1)$ & $\delta_{0,\cycle_{12,3}(i+1)}$ & $\pos_{12,3}(i+1)$ & $\rho_{12,3}(i+1)$\\ \hline
		0   &  1				   & 0							     &  0				 & 0					\\
		1   &  2				   & 0							     &  0				 & 0					\\
		2   &  0 				   & 1							     &  1				 & $-a^2b$					\\
		3   &  1 				   & 0								 &  1				 & 0		   	 		\\
		4   &  2 				   & 0								 &  1				 & 0					\\
		5   &  0 				   & 1							     &  2				 & $ab^2$					\\
		6   &  1 				   & 0							     &  2				 & 0					\\
		7   &  2 				   & 0							     &  2				 & 0					\\
		8   &  0 				   & 1							     &  3				 & $-b^3$					\\
		9   &  1 				   & 0							     &  3				 & 0					\\
		10  &  2 				   & 0							     &  3				 & 0					\\
		11  &  0 				   & 1							     &  0				 & $a^3$					\\ \hline
	\end{tabular}
\end{center}
Therefore, \[\left(aP_{12}+bP_{12}^{4}\right)^{-1}=\dfrac{1}{a^{4}-b^{4}} \Circulante \left(0,0,-ba^2,0,0,b^2a,0,0,-b^3,0,0,a^3\right).\]
%%%%%%%%%%%%%%%%%%%%%%%%%%%%%%%%%%%%%%%%%%%%%%%%%%%%%%%%%%%%%%%%%%%%%%%%%%%%%%%
%
%%%%%%
%
%%%%%%%%%%%%%%%%%%%%%%%%%%%%%%%%%%%%%%%%%%%%%%%%%%%%%%%%%%%%%%%%%%%%%%%%%%%%%%%
\end{example}

\medskip

\section{Drazin inverse of \(aP_{n}^{s_{1}}+bP_{n}^{s_{2}}\)}\label{Sec_Inv_Dra_aP_bP}
Given a matrix \(A\), the column space of \(A\) is denoted by \(\Rank{A}\) and its dimension by \(\rank{A}\). The null space of \(A\) is denoted by \(\Null{A}\) and its dimension, called nullity, by \(\dnull{A}\). 
The index of a square matrix \(A\), denoted by $\Ind{A}$, is the smallest non-negative integer $k$ for which $ \Rank{A^{k}} = \Rank{A^{k+1}}$. 
It is well-known that a circulant matrix has index 0 or 1, see \cite{davis2012circulant}.  
Let $A$ be a matrix of index $k$, 
the Drazin inverse of $A$, denoted by  \(A^{D}\), is the unique matrix such that	
\begin{enumerate}
	\item $AA^{D}=A^{D}A$,
	\item $A^{k+1}A^{D}=A^{k}$,
	\item $A^{D}AA^{D}=A^{D}$.
\end{enumerate}

%Note that the Drazin inverse of a matrix of index 1 equals its group inverse (that is denoted, for a matrix \(A\), by \(A^{\#}\)), see \cite{ben2003generalized}.

In  \cite{encinas2019Circulant} was proved the following Bjerhammar-type condition for the Drazin inverse. We find it handy for checking the Drazin conditions in combinatorial settings.

\begin{theorem}[\cite{encinas2019Circulant}] \label{Teo_JaumePastine}
	Let $A$ and $D$ be square matrices of order \(n\), with $ \Ind{A} = k$, such that $\Null{A^{k}} = \Null{D} $, and $AD=DA$. Then $A^{k+1}D=A^{k}$ if and only if $D^{2}A=D$.
\end{theorem}

\begin{theorem} \label{Drazin_General}
	Let $n$, $s_{1}$ and $s_{2}$ be integers such that $ 0 \leq s_{1} <s_{2} < n$. Let $a$ and  $b$ be non-zero complex numbers such that \(a^{n \backslash \left(s_{2}-s_{1}\right)}-\left( -b\right) ^{n \backslash \left( s_{2}-s_{1}\right)} = 0\). Then \(\left(aP_{n}^{s_1}+bP_{n}^{s_{2}}\right)^{D}\) equals
	\begin{equation}
		P_{n}^{n-s_{1}} P_{\sigma_{n,s_{2}-s_{1}}}
		\left[
		I_{(n,s_{2}-s_{1})} \otimes \left(aI_{n \backslash \left(s_{2}-s_{1}\right) } +bP_{n \backslash \left(s_{2}-s_{1}\right) }\right)^{D}
		\right]
		P_{\sigma_{n,s_{2}-s_{1}}}^{T}.
	\end{equation}
\end{theorem}
\begin{proof}
	The following two facts are well-known. If \(A=XBX^{-1}\), then \(A^{D}=XB^{D}X^{-1}\). If \(AB=BA\), then \(\left(AB\right)^{D}=B^{D}A^{D}=A^{D}B^{D}\). See \cite{ben2003generalized} or \cite{campbell2009generalized}.
		
	In 1997, Wang proved that 
	\(\left(A\otimes B\right)^{D}=A^{D}\otimes B^{D}\) and \(\Ind{A\otimes B}=\max\{\Ind{A},\Ind{B}\}\), see Theorem 2.2. of \cite{wang1997weighted}.
	Let \(s=s_2-s_1\). Then, by Theorem \ref{Untagling_theo}
%	\begin{equation*}
%	\left(aP_{n}^{s_1}+bP_{n}^{s_{2}}\right)^{D} =  P_{n}^{n-s_{1}} \left(aI_{n}+bP_{n}^{s}\right)^{D}.
%	\end{equation*}
%	
%	\begin{equation*}
%	P_{\sigma_{n,s_{2}-s_{1}}}^{T}
%	P_{n}^{n-s_{1}} \left[aI_{n}+bP_{n}^{s}	
%	\right]
%	P_{\sigma_{n,s_{2}-s_{1}}} =I_{(n,s)} \otimes \left(aI_{n \backslash s } +bP_{n \backslash s } \right)
%	\end{equation*}
%	
%	\begin{equation*}
%	\left[ I_{(n,s)} \otimes \left(aI_{n \backslash s } +bP_{n \backslash s } \right) \right]^{D}
%	= I_{(n,s)} \otimes \left(aI_{n \backslash s } +bP_{n \backslash s } \right)^{D}
%	\end{equation*}
%
%	\begin{equation*}
%	\left[ I_{(n,s)} \otimes \left(aI_{n \backslash s } +bP_{n \backslash s } \right) \right]^{D}=
%	P_{\sigma_{n,s}}^{T}
%	\left[
%	P_{n}^{n-s_{1}}
%	\left(
%	aI_{n}+bP_{n}^{s}
%	\right)	
%	\right]^{D}
%	P_{\sigma_{n,s}}
%	\end{equation*}
%	
%	\begin{equation*}
%	I_{(n,s)} \otimes \left(aI_{n \backslash s } +bP_{n \backslash s }\right)^{D}
%	\end{equation*}
	\begin{align*}
	\left(aP_{n}^{s_1}+bP_{n}^{s_{2}}\right)^{D} 
	&=P_{n}^{n-s_{1}}
	\left[
	P_{\sigma_{n,s}}
	\left[
	I_{(n,s)} \otimes \left(aI_{n \backslash s } +bP_{n \backslash s }\right)
	\right]
	P_{\sigma_{n,s}}^{T}
	\right]^{D}.\\
	{} &= 
	P_{n}^{n-s_{1}} 
	P_{\sigma_{n,s}}
	\left[
	I_{(n,s)} \otimes \left(aI_{n \backslash s } +bP_{n \backslash s }\right)
	\right]^{D}
	P_{\sigma_{n,s}}^{T}\\
	{} &=
	P_{n}^{n-s_{1}} 
	P_{\sigma_{n,s}}
	\left[
	I_{(n,s)} \otimes \left(aI_{n \backslash s } +bP_{n \backslash s }\right)^{D}
	\right]
	P_{\sigma_{n,s}}^{T}. \qedhere
	\end{align*}
\end{proof}

Therefore, we just need to study the Drazin inverse of the matrices of the form \(aI_{n}+bP_{n}\). 

By Corollary \ref{Teo_Determinante}, we know that $aP_{n}^{s_{1}}+bP_{n}^{s_{2}} $ is singular if and only if \(a^{n \backslash s}-\left( -b\right) ^{n \backslash  s}=0\), where \(s=s_{2}-s_{1}\). If \(n \backslash  s\) is even, then \(a^{n \backslash s}-(-b)^{n \backslash s}= 0\) if and only if \(|a|=|b|\). If \(n \backslash s\) is odd, then \(a^{n \backslash s}-(-b)^{n \backslash s}=0\) if and only if \(a=-b\). Hence, the matrices of the form \(a\left(I_{n}-P_{n} \right)\) are always singular, but the matrices of form \(a\left(I_{n}+P_{n} \right)\) are singular if and only if \(n\) is even. All of these facts reduce the calculus of Drazin inverse of singular circulant matrices of form \(aP_{n}^{s_{1}}+bP_{n}^{s_{2}}\) to the calculus of Drazin inverse of just two matrices: \(I_{n}-P_{n}\) and \(I_{2n}+P_{2n}\).
%%%%%%%%%%%%%%%%%%%%%%%%%%%%%%%%%%%%%%%%%%%%%%%%%%%%%%%%%%%%%%%%%%%%%%%%%%%%%%%
%
%%%%%%
%
%%%%%%%%%%%%%%%%%%%%%%%%%%%%%%%%%%%%%%%%%%%%%%%%%%%%%%%%%%%%%%%%%%%%%%%%%%%%%%%

\subsection{Drazin inverse of \(I_{n}-P_{n}\)}

\begin{theorem} \label{teo_drazin_1}
Let \(n\) be a non-negative integer. The Drazin inverse of \(aI_{n}-aP_{n}\) is
	\begin{equation}\label{Eq_teo_drazin_1}
	\left(I_{n}-P_{n}\right)^{D}:=\dfrac{1}{2\,n}\sum_{i=0}^{n-1} \left(n-2\,i-1\right) P_{n}^{i}.
	\end{equation}
\end{theorem}
\begin{example}
Let $a$ be a non-zero complex number. Thus, we have that
	\[\left(aI_{4}-aP_{4}\right)^D=\tfrac{1}{8a}\Circulante \left(3,1,-1,-3\right),\]
	and
	\[\left(aI_{5}-aP_{5}\right)^D=\tfrac{1}{10a}\Circulante \left(4,2,0,-2,-4\right). \]
\end{example}

In order to prove this theorem, we will use some polynomial tools. The idea is to prove that both matrices in \eqref{Eq_teo_drazin_1} have the same null space, and then use the Theorem \ref{Teo_JaumePastine}.

By \(1\!\!1_{n}\), we denoted the vector of all ones. The easy proof of the following lemma is left to the reader. 
\begin{lemma}
Let \(n\) be a non-negative integer. Then
	\[
	1\!\!1_{n} \in \Null{I_{n}-P_{n}} \cap \Null{\dfrac{1}{2\,n}\sum_{i=0}^{n-1} \left(n-2\,i-1\right) P_{n}^{i}}.
	\]
\end{lemma}

Every circulant matrix \(\Circulante \left(c_{0},\dots,c_{n-1}\right)\) has its associated polinomial \(P_{C}(x)=\sum_{i=0}^{n-1} c_{i}x^{i}\). Ingleton, in 1955, proved the following proposition.

\begin{proposition}[Proposition 1.1 in \cite{ingleton1956rank}]\label{prop_ingleton}
	The rank of a circulant matrix \(C\) of order \(n\) is \(n-d\), where \(d\) is the degree of the greatest common divisor of \(x^{n}-1\) and the associated polynomial of \(C\).
\end{proposition}

%This mean that to know how many unit roots are roots of the associated polynomial of a circulant matrix is to know the rank of the circulant matrix.

This means that in order to know the rank of some circulant matrix, it is enough to know how many $n$-th roots of the unit are roots of its associated polynomial.

%%%%%%%%%%%%%%%%%%%%%%%%%%%%%%%%%%%%%%%%%%%%%%%%%%%%%%%%%%%%%%%%%%%%%%%%%%%%%%%%%%%%%%%%%%%%%%%%%%%%%%%%%%
Let us consider the polynomials
\begin{equation}\label{Poly_1}
H_{j}^n(x)=\sum_{i=0}^{n-1} \left(n-2 \, (i+j)-1\right)\,x^i,
\end{equation}
for $j=0,\ldots,n-1$. Notice that  $H_{0}^n(x)$ is the associated polynomial of the circulant matrix $\sum_{i=0}^{n-1} \left(n-2 \, i-1\right) P_{n}^{i}$.

We denote by 
$$\Omega_{\ell}=\{x\in \mathbb{C}: x^{\ell}=1\},$$ 
the set of $\ell$-th root of the unit. 
Notice that, $$\Omega_{n}=\{\omega_{n}^k: k=0,\ldots,n-1\},$$ 
where \(\omega=\exp (\frac{2\pi i}{n})\), 
i.e. \( \,\omega_{n}^k \) are exactly the roots of the polynomial  $p(x)=x^n-1$.
Moreover, they are all simple roots of $p(x)=x^n-1$, because they are all different and $p$ has degree $n$.  
Moreover, if we denote
$$
\Phi_{n}(x):=\dfrac{x^n-1}{x-1}=\sum_{i=0}^{n-1}x^{i},
$$
then $\omega_{n}^k$ are all simple roots of $\Phi_{n}$ for any $k\ne 0$. In terms of derivatives, this means that 
$$
\Phi_{n}(\omega)=0\quad \text{ and } \quad \Phi_{n}'(\omega)\ne 0,
$$
for all $\omega \in \Omega_{n}\smallsetminus\{1\}$.

\begin{proposition}\label{Denis1}
Let \(n\) be a non-negative integer. Then
	$$H_{j}^n(\omega)\ne 0,$$
	for all $j=0,\ldots,n-1$ and $\omega \in \Omega_{n}\smallsetminus\{1\}$.
\end{proposition} 
\begin{proof}
	We have that 
	$$H_{j}^n(x)=\sum_{i=0}^{n-1} \left(n-2 \, (i+j)-1\right) \, x^i= (n-2 \, j-1)\sum_{i=0}^{n-1} x^i-2\sum_{i=0}^{n-1}i \, x^{i}.$$
	Notice that $\sum_{i=0}^{n-1}i \, x^{i}=x \,\Phi'_n(x)$, and so we obtain that
	\begin{equation}\label{eq Pj-Phi}
	H_{j}^{n}(x)=(n-2 \, j-1) \, \Phi_{n}(x)- 2 \, x \, \Phi'_{n}(x).
	\end{equation}
	Now, assume that $H_{j}^n(\omega)=0$ for some $\omega\in \Omega_{n}\smallsetminus\{1\}$, thus \eqref{eq Pj-Phi} implies that
	$$0=(n-2 \, j-1) \, \Phi_{n}(\omega)-2 \, \omega \, \Phi'_{n}(\omega).$$
	Hence, we have that $\omega \, \Phi'_{n}(\omega)=0$, since $\Phi_{n}(\omega)=0$. But this cannot occur since $\omega\ne 0$ and 
	$\omega$ are all simple roots of  $\Phi_{n}$. Therefore $H_{j}^n(\omega)\ne 0$
	for all $j=0,\ldots,n-1$ and $\omega \in \Omega_{n}\smallsetminus\{1\}$, as asserted.
\end{proof}

We are ready to prove the main result of the subsection.
\begin{proof}
	\textit{\!of Theorem \ref{teo_drazin_1}: }Since index of all singular circulant matrices is $1$ and the rank of \(I_{n}-P_{n}\) is clearly \(n-1\), by Propositions \ref{prop_ingleton} and \ref{Denis1}, we have that
	\[
	\Null{I_{n}-P_{n}} = \Null{\dfrac{1}{2 \, n}\sum_{i=0}^{n-1} \left(n-2 \, i-1\right) P_{n}^{i}}.
	\]
	By Theorem \ref{Teo_JaumePastine}, we just need to check that
	\[
	\left(I_{n}-P_{n}\right)^2\left(\dfrac{1}{2 \, n}\sum_{i=0}^{n-1} \left(n-2 \, i-1\right) P_{n}^{i}\right)=I_{n}-P_{n}.
	\]
	which is left to the reader.
\end{proof}

In order to obtain an explicit formula for the Drazin inverse of the circulant matrices of the form \(P_{n}^{s_1}-P_{n}^{s_2}\) we define
\begin{equation}\label{Ecuacion_Rho_1}
\rho_{n,s}^{*}(i):=\delta_{0,\cycle_{n,s}(i)} \, \left(n\backslash s-2  \, \pos_{n,s}(i)-1\right),
\end{equation}
where \(\delta\) is the usual Kronecker delta. In this case, $\rho^*$ satisfies the following:
\begin{enumerate}
    \item For $n>0$, \(\rho_{n,1}^{*}(i)=(n-2i-1)\) for all \(i=0,\dots,n-1\).     %For $n>0$, $\rho_{n,1}^{*}(i)=(n-2i-1)$ for all $i=0,\dots,\n-1 $.
    \item For $0<s<n$, $\rho_{n \backslash s,1}^{*}(i)= \rho_{n,s}^{*}(i \, s)$ for all $i=0,\ldots,(n\backslash s)-1$. 
\end{enumerate}

% (-1)^{\pos_{n,s}(i)}\,
\begin{corollary}\label{coro_Drazin_1}
Let $n$, $s_{1}$ and $s_{2}$ be integers such that $ 0 \leq s_{1} <s_{2} < n$. Then 
\begin{equation}
\left(P_{n}^{s_1}-P_{n}^{s_{2}}\right)^{D} =
\dfrac{1}{2 \, (n \backslash\left(s_{2}-s_{1}\right))} \sum_{i=0}^{n-1} \rho_{n,s_{2}-s_{1}}^{*}(i+s_{1})P_{n}^{i}.
\end{equation}
\end{corollary}

\begin{proof}
    Let \(s=s_{2}-s_{1}\). By Theorem \ref{Drazin_General} we have that \(\left(P_{n}^{s_{1}}-P_{n}^{s_{2}} \right)^{D}\) is equal to
    \[
    P_{n}^{n-s_{1}} P_{\sigma_{n,s}}\left[I_{(n,s)} \otimes \left(I_{n \backslash s} -P_{n \backslash s} \right)^{D} \right] P_{\sigma_{n,s}}^{T}.
    \]
    Thus, by Theorem \ref{teo_drazin_1} and  \eqref{Ecuacion_Rho_1}
    \[
    \left(P_{n}^{s_{1}}-P_{n}^{s_{2}} \right)^{D} = P_{n}^{n-s_{1}} P_{\sigma_{n,s}} \left( \sum\limits_{i=0}^{(n\backslash s) - 1} I_{(n,s)} \otimes \frac{\rho_{n\backslash s,1}^{*}(i)}{2 \, (n\backslash s)} P_{n\backslash s}^{i} \right) P_{\sigma_{n,s}}^{T}.
    \]
    Hence, by Corollary \ref{coro_J1}, and expressions \eqref{Ecuacion_Potencias_P_ns1} and  \eqref{Ecuacion_Potencias_P_ns2}
    \begin{align*}
        \left(P_{n}^{s_{1}}-P_{n}^{s_{2}} \right)^{D} &= \sum\limits_{i=0}^{(n\backslash s) - 1} \frac{\rho_{n\backslash s,1}^{*}(i)}{2 \, (n\backslash s)} P_{n}^{(i \, s) - s_{1}} \\
        &= \sum\limits_{i=0}^{(n\backslash s) - 1} \frac{\rho_{n,s}^{*}(i \, s)}{2 \, (n\backslash s)} P_{n}^{(i \, s) - s_{1}} \\
        &= \sum\limits_{j=0}^{n-1} \frac{\rho_{n,s}^{*}(j+s_{1})}{2 \, (n\backslash s)} P_{n}^{j}. \qedhere
        \end{align*}
\end{proof}

\begin{example}
For example, we compute \(\left(aP_{12}-aP_{12}^{4}\right)^{D}\). Since \(12 \backslash 3=4\). By Theorems \ref{Drazin_General} and \ref{teo_drazin_1}, we know that the Drazin inverse is essentially 3 blocks of \(\Circulante \left(3,1,-1,-3\right)\) merge in a \(12\times 12\) matrix via \(P_{12}\) and \(P_{\sigma_{12,3}}\). 
\begin{center}
	\begin{tabular} {|c|c|c|c|c|}
		\hline
		$i$ & $\cycle_{12,3}(i+1)$ & $\delta_{0,\cycle_{12,3}(i+1)}$ & $\pos_{12,3}(i+1)$ & $\rho_{12,3}^{*}(i+1)$\\ \hline
		0   &  1				   & 0							     &  0				 & 0					\\
		1   &  2				   & 0							     &  0				 & 0					\\
		2   &  0 				   & 1							     &  1				 & 1					\\
		3   &  1 				   & 0								 &  1				 & 0		   	 		\\
		4   &  2 				   & 0								 &  1				 & 0					\\
		5   &  0 				   & 1							     &  2				 & -1					\\
		6   &  1 				   & 0							     &  2				 & 0					\\
		7   &  2 				   & 0							     &  2				 & 0					\\
		8   &  0 				   & 1							     &  3				 & -3					\\
		9   &  1 				   & 0							     &  3				 & 0					\\
		10  &  2 				   & 0							     &  3				 & 0					\\
		11  &  0 				   & 1							     &  0				 & 3					\\ \hline
	\end{tabular}
\end{center}

Therefore, \(\left(aP_{12}-aP_{12}^{4}\right)^{D}=\dfrac{1}{8 \, a} \Circulante \left(0,0,1,0,0,-1,0,0,-3,0,0,3\right)\).
%%%%%%%%%%%%%%%%%%%%%%%%%%%%%%%%%%%%%%%%%%%%%%%%%%%%%%%%%%%%%%%%%%%%%%%%%%%%%%%%%%%%%%%%%%%%%%%%%%%%%%%%%%
%%%%%%%%%%%%%%%%%%%%%%%%%%%%%%%%%%%%%%%%%%%%%%%%%%%%%%%%%%%%%%%%%%%%%%%%%%%%%%%
%
%%%%%%
%
%%%%%%%%%%%%%%%%%%%%%%%%%%%%%%%%%%%%%%%%%%%%%%%%%%%%%%%%%%%%%%%%%%%%%%%%%%%%%%%
\end{example}	

\subsection{Drazin inverse of \(I_{2n}+P_{2n}\)}

\begin{theorem} \label{teo_drazin_2}
Let \(n\) be a non-negative integer. The Drazin inverse of \(I_{2n}+P_{2n}\) is 
\begin{equation}\label{Eq_teo_drazin_2}
\left(I_{2n}+P_{2n}\right)^{D}=\frac{1}{4 \, n}\sum_{i=0}^{2 \, n-1} (-1)^{i}\left(2 \, n-2 \, i-1\right) P_{2n}^{i}.
\end{equation}
\end{theorem}
%\textcolor{blue}{\begin{example}
%	Thus, 
%	\[\left(3I_{4}+3P_{4}\right)^D=\frac{1}{24}\Circulante \left(3,-1,-1,3\right),\]
%\end{example}
%}
The idea of the proof is the same as before, i.e. we want to show that both matrices in \eqref{Eq_teo_drazin_2} have the same null space, and then use Theorem \ref{Teo_JaumePastine}.

We define \(\pm1\!\!1_{n}=((-1)^i)_{i=0}^{2n-1}\in \mathbb{R}^{2n}\), e particular vector of ones and minus ones. 
The following lemma, can be proved easily and it is left to the reader.  
\begin{lemma} \label{Lema_null_2}
Let \(n\) be a non-negative integer. Then
\[
\pm1\!\!1_{n} \in \Null{I_{2n}+P_{2n}} \cap \Null{\dfrac{1}{4 \,n}\sum_{i=0}^{2n-1} (-1)^{i}\left(2 \, n-2 \, i-1\right) P_{2n}^{i}}.
\]
\end{lemma}

%%%%%%%%%%%%%%%%%%%%%%%%%%%%%%%%%%%%%%%%%%%%%%%%%%%%%%%%%%%%%%%%%%%%%%%%%%%%%%%%%%%%%%%%%%%%%%%%%%%%%%%%%%
Let us consider the polynomials
\begin{equation}
\widehat{H}_{j}^{2n}(x)=\sum_{i=0}^{2n-1} (-1)^{i+j}\left(2 \, n-2 \, (i+j)-1\right)\,x^i,
\end{equation}
for $j=0,\ldots,2\,n-1$. As before, notice that $\widehat{H}_{0}^{2n}(x)$ is the associated polynomial of the circulant matrix $\sum_{i=0}^{2n-1} (-1)^{i}\left(2 \, n-2 \, i-1\right) P_{2n}^{i}$ the following result assures that its rank is $2\,n-1$.

\begin{proposition} \label{Denis2}
	Let $n$ be a non-negative integer. Then 
	$$\widehat{H}_{j}^{2n}(\omega)\ne 0$$
	for all $j=0,\ldots,2\,n-1$ and $\omega\in \Omega_{2n}\smallsetminus \{-1\}$.
\end{proposition}
\begin{proof}
	On the first hand, the linear transformation $T(x)=-x$ for $x\in \mathbb{C}$ is a permutation on $\Omega_{2n}$, i.e.\@ we have that
	\begin{equation}\label{Invarianza -}
	\omega \text{ is a $(2\,n)$-th root of $1$}\quad \Longleftrightarrow\quad -\omega \text{ is a $(2\,n)$-th root of $1$}.
	\end{equation}
	On the other hand, if $C_{i+j}=2\,n-2\,(i+j)-1$ we have that 
	$$\widehat{H}_{j}^{2n}(x)=\sum_{i=0}^{2n-1} (-1)^{i+j} C_{i+j}x^{i}=(-1)^{j}\sum_{i=0}^{2n-1} C_{i+j}(-x)^{i}=(-1)^{j}H_{j}^{2n} (-x),$$
	where $H_{j}^{2n}$ as in \eqref{Poly_1}. By \eqref{Invarianza -} and Proposition \ref{Denis1} we obtain that
	$$\widehat{H}_{j}^{2n}(\omega)=0 \quad \Longleftrightarrow\quad H_{j}^{2n}(-\omega)=0 \quad \Longleftrightarrow\quad \omega=-1.$$
	So, if $\omega\ne -1$, then $\widehat{H}_{j}^{2n}(\omega)\ne 0$, as desired.
\end{proof}

We are ready to prove the main result of the subsection.
\begin{proof}
\textit{of Theorem \ref{teo_drazin_2}: } The rank of \(I_{2n}+P_{2n}\) is clearly \(2\,n-1\). Hence, by Lemma  \ref{Lema_null_2} and Proposition \ref{Denis2}
\[
\Null{I_{2n}+P_{2n}} = \Null{\dfrac{1}{4\,n}\sum_{i=0}^{2n-1} (-1)^{i}\left(2\,n-2\,i-1\right) P_{2n}^{i}}.
\]
By Theorem \ref{Teo_JaumePastine}, it suffices to check that
\[
\left(I_{2n}+P_{2n}\right)^2\left(\dfrac{1}{4\,n}\sum_{i=0}^{2n-1} (-1)^{i} \left(2\,n-2\,i-1\right) P_{2n}^{i}\right)=I_{2n}+P_{2n}.
\]
which is left to the reader.
\end{proof}

\begin{corollary}\label{coro_Drazin_2}
Let $n$, $s_{1}$, and $s_{2}$ be integers such that $ 0 \leq s_{1} <s_{2} < 2n$. Then 
\begin{equation}
\left(P_{2n}^{s_1}+P_{2n}^{s_{2}}\right)^{D} = \dfrac{1}{2\,(2\,n \backslash\left(s_{2}-s_{1}\right))} \sum_{i=0}^{2n-1} (-1)^{i+s_{1}}\rho_{2n,s_{2}-s_{1}}^{*}(i+s_{1})P_{2n}^{i}.
\end{equation}
\end{corollary}

\begin{example}
Let us compute \(\left(aP_{12}+aP_{12}^{4}\right)^{D}\). Since \(12 \backslash 3=4\). By Theorems \ref{Drazin_General} and \ref{teo_drazin_1}, we know that the Drazin inverse is essentially 3 blocks of \(\Circulante \left(3,-1,-1,3\right)\) merge in a \(12\times 12\) matrix via \(P_{12}\) and \(P_{\sigma_{12,3}}\). Since \(4-1=3\), then \(12\backslash 3=4\), and \((12,3)=3\). Therefore, we have
\begin{center}
	\begin{tabular} {|c|c|c|c|c|}
		\hline
		$i$ & $\cycle_{12,3}(i+1)$ & $\delta_{0,\cycle_{12,3}(i+1)}$ & $\pos_{12,3}(i+1)$& $(-1)^{i+1}\rho_{12,3}^{*}(i+1)$\\ \hline
		0   &  1				   & 0							     &  0				 & 0					\\
		1   &  2				   & 0							     &  0				 & 0					\\
		2   &  0 				   & 1							     &  1				 & -1					\\
		3   &  1 				   & 0								 &  1				 & 0		   	 		\\
		4   &  2 				   & 0								 &  1				 & 0					\\
		5   &  0 				   & 1							     &  2				 & -1					\\
		6   &  1 				   & 0							     &  2				 & 0					\\
		7   &  2 				   & 0							     &  2				 & 0					\\
		8   &  0 				   & 1							     &  3				 & 3					\\
		9   &  1 				   & 0							     &  3				 & 0					\\
		10  &  2 				   & 0							     &  3				 & 0					\\
		11  &  0 				   & 1							     &  0				 & 3					\\ \hline
	\end{tabular}
\end{center}
Therefore, \(\left(aP_{12}+aP_{12}^{4}\right)^{D}=\dfrac{1}{8\,a} \Circulante \left(0,0,-1,0,0,-1,0,0,3,0,0,3\right)\).
\end{example}

\medskip

\section{Block circulant matrices with two parameters}\label{Sec_Blck_Circ}

Let $n,s_{1}$ and $s_{2}$ be non-negative integers such that $0\leq s_{1} < s_{2} \leq n-1$. Let $A$ and  $B$ be two square matrices of order \(r\). We will use the basic properties of the Kronecker product, see \cite{steeb2011matrix}. The matrices of the form \(P^{s_{1}}_{n}\otimes A + P^{s_{2}}_{n} \otimes B\) we call block circulant matrices with two parameters.

For example, for the matrices
\begin{equation} \label{Ejem_Circ_Block1}
	A =
	\left[
	\begin{array}{cc}
		1 & 2 \\
		2 & 1
	\end{array}
	\right]
	\quad \text{and} \quad
	B =
	\left[
	\begin{array}{rr}
		-1 & 3 \\ 
		3 & -1
	\end{array}
	\right],
\end{equation}
we have that \(P_{8} \otimes A + P^{3}_{8} \otimes B\) is
\begin{equation*}\label{Matriz_Ejemplo1}
	\left[
	\begin{array}{rr|rr|rr|rr|rr|rr|rr|rr}
		0 & 0 & \textcolor{red}{\textbf{1}} & \textcolor{red}{\textbf{2}} & 0 & 0 & \textcolor{blue}{\textbf{-1}} & \textcolor{blue}{\textbf{3}} & 0 & 0 & 0 & 0 & 0 & 0 & 0 & 0 \\
		0 & 0 & \textcolor{red}{\textbf{2}} & \textcolor{red}{\textbf{1}} & 0 & 0 & \textcolor{blue}{\textbf{3}} & \textcolor{blue}{\textbf{-1}} & 0 & 0 & 0 & 0 & 0 & 0 & 0 & 0 \\
		\hline
		0 & 0 & 0 & 0 & \textcolor{red}{\textbf{1}} & \textcolor{red}{\textbf{2}} & 0 & 0 & \textcolor{blue}{\textbf{-1}} & \textcolor{blue}{\textbf{3}} & 0 & 0 & 0 & 0 & 0 & 0 \\
		0 & 0 & 0 & 0 & \textcolor{red}{\textbf{2}} & \textcolor{red}{\textbf{1}} & 0 & 0 & \textcolor{blue}{\textbf{3}} & \textcolor{blue}{\textbf{-1}} & 0 & 0 & 0 & 0 & 0 & 0 \\
		\hline
		0 & 0 & 0 & 0 & 0 & 0 & \textcolor{red}{\textbf{1}} & \textcolor{red}{\textbf{2}} & 0 & 0 & \textcolor{blue}{\textbf{-1}} & \textcolor{blue}{\textbf{3}} & 0 & 0 & 0 & 0 \\
		0 & 0 & 0 & 0 & 0 & 0 & \textcolor{red}{\textbf{2}} & \textcolor{red}{\textbf{1}} & 0 & 0 & \textcolor{blue}{\textbf{3}} & \textcolor{blue}{\textbf{-1}} & 0 & 0 & 0 & 0 \\
		\hline
		0 & 0 & 0 & 0 & 0 & 0 & 0 & 0 & \textcolor{red}{\textbf{1}} & \textcolor{red}{\textbf{2}} & 0 & 0 & \textcolor{blue}{\textbf{-1}} & \textcolor{blue}{\textbf{3}} & 0 & 0 \\
		0 & 0 & 0 & 0 & 0 & 0 & 0 & 0 & \textcolor{red}{\textbf{2}} & \textcolor{red}{\textbf{1}} & 0 & 0 & \textcolor{blue}{\textbf{3}} & \textcolor{blue}{\textbf{-1}} & 0 & 0 \\
		\hline
		0 & 0 & 0 & 0 & 0 & 0 & 0 & 0 & 0 & 0 & \textcolor{red}{\textbf{1}} & \textcolor{red}{\textbf{2}} & 0 & 0 & \textcolor{blue}{\textbf{-1}} & \textcolor{blue}{\textbf{3}} \\
		0 & 0 & 0 & 0 & 0 & 0 & 0 & 0 & 0 & 0 & \textcolor{red}{\textbf{2}} & \textcolor{red}{\textbf{1}} & 0 & 0 & \textcolor{blue}{\textbf{3}} & \textcolor{blue}{\textbf{-1}} \\
		\hline
		\textcolor{blue}{\textbf{-1}} & \textcolor{blue}{\textbf{3}} & 0 & 0 & 0 & 0 & 0 & 0 & 0 & 0 & 0 & 0 & \textcolor{red}{\textbf{1}} & \textcolor{red}{\textbf{2}} & 0 & 0 \\
		\textcolor{blue}{\textbf{3}} & \textcolor{blue}{\textbf{-1}} & 0 & 0 & 0 & 0 & 0 & 0 & 0 & 0 & 0 & 0 & \textcolor{red}{\textbf{2}} & \textcolor{red}{\textbf{1}} & 0 & 0 \\
		\hline
		0 & 0 & \textcolor{blue}{\textbf{-1}} & \textcolor{blue}{\textbf{3}} & 0 & 0 & 0 & 0 & 0 & 0 & 0 & 0 & 0 & 0 & \textcolor{red}{\textbf{1}} & \textcolor{red}{\textbf{2}}\\
		0 & 0 & \textcolor{blue}{\textbf{3}} & \textcolor{blue}{\textbf{-1}} & 0 & 0 & 0 & 0 & 0 & 0 & 0 & 0 & 0 & 0 & \textcolor{red}{\textbf{2}} & \textcolor{red}{\textbf{1}} \\
		\hline
		\textcolor{red}{\textbf{1}} & \textcolor{red}{\textbf{2}} & 0 & 0 & \textcolor{blue}{\textbf{-1}} & \textcolor{blue}{\textbf{3}} & 0 & 0 & 0 & 0 & 0 & 0 & 0 & 0 & 0 & 0 \\
		\textcolor{red}{\textbf{2}} & \textcolor{red}{\textbf{1}} & 0 & 0 & \textcolor{blue}{\textbf{3}} & \textcolor{blue}{\textbf{-1}} & 0 & 0 & 0 & 0 & 0 & 0 & 0 & 0 & 0 & 0
	\end{array}
	\right]
\end{equation*}

The following tell us that is enough to study the circulant matrices of the form \(I_{n}\otimes A + P^{s}_{n} \otimes B\). We have that
\begin{equation}\label{Ecuacion_C_nqs_Cns_q}
	P^{s_{1}}_{n}\otimes A + P^{s_{2}}_{n} \otimes B = P_{n}^{s_{1}} \otimes I_{r} \left(I_{n}\otimes A + P^{s_{2}-s_{1}}_{n} \otimes B\right).
\end{equation}
Thus, if \(P_{n}^{s_{1}}\otimes A +P_{n}^{s_{2}}\otimes B\) is non-singular, then
\begin{equation}\label{Corolario_Inversa}
	(P_{n}^{s_{1}}\otimes A +P_{n}^{s_{2}}\otimes B)^{-1} = (I_{n}\otimes A + P_{n}^{s_{2}-s_{1}} \otimes B)^{-1}  P_{n}^{n-s_{1}} \otimes I_{r}.
\end{equation}

The next is a direct consequence of the following fact: \(i-j = n-s  \!\mod n\) if and only if \(j-i= s \!\mod n\).
\begin{equation} \label{Lema_C_{n,n-j}=C^{T}_{n,j}}
	\left(I_{n}\otimes A + P^{s}_{n} \otimes B \right)^{T} = I_{n}\otimes A^{T} + P^{n-s}_{n} \otimes B^{T}.
\end{equation}

Let \(s=s_{2}-s_{1}\). Note that by Theorem \ref{Untagling_theo} we have that the matrix \(I_{(n,s)} \otimes \left(I_{n\backslash s} \otimes A + P_{n\backslash s} \otimes B\right)\) is equal to
\begin{equation}\label{Ecuacion_Factorizacion_Bloques}
\left(P_{\sigma_{n,s}}^{T} \otimes I_{r}\right) \left(P_{n}^{n-s_{1}} \otimes I_{r} \right) \left(P_{n}^{s_{1}} \otimes A + P^{s_{2}}_{n} \otimes B\right) \left(P_{\sigma_{n,s}} \otimes I_{r}\right).
\end{equation}

For the matrices given in \ref{Ejem_Circ_Block1} we have that
\[\left(P_{\sigma_{8,2}}^{T} \otimes I_{2}\right) \left(P_{8}^{7} \otimes I_{2}\right) \left(P_{8} \otimes A + P^{3}_{8} \otimes B\right) \left(P_{\sigma_{8,2}} \otimes I_{2} \right) =\]
\[
\left[
\begin{array}{rrrrrrrr|rrrrrrrr}
\textcolor{red}{\textbf{1}} & \textcolor{red}{\textbf{2}} & \textcolor{blue}{\textbf{-1}} & \textcolor{blue}{\textbf{3}} & 0 & 0 & 0 & 0 & 0 & 0 & 0 & 0 & 0 & 0 & 0 & 0 \\
\textcolor{red}{\textbf{2}} & \textcolor{red}{\textbf{1}} & \textcolor{blue}{\textbf{3}} & \textcolor{blue}{\textbf{-1}} & 0 & 0 & 0 & 0 & 0 & 0 & 0 & 0 & 0 & 0 & 0 & 0 \\
0 & 0 & \textcolor{red}{\textbf{1}} & \textcolor{red}{\textbf{2}} & \textcolor{blue}{\textbf{-1}} & \textcolor{blue}{\textbf{3}} & 0 & 0 & 0 & 0 & 0 & 0 & 0 & 0 & 0 & 0 \\
0 & 0 & \textcolor{red}{\textbf{2}} & \textcolor{red}{\textbf{1}} & \textcolor{blue}{\textbf{3}} & \textcolor{blue}{\textbf{-1}} & 0 & 0 & 0 & 0 & 0 & 0 & 0 & 0 & 0 & 0 \\
0 & 0 & 0 & 0 & \textcolor{red}{\textbf{1}} & \textcolor{red}{\textbf{2}} & \textcolor{blue}{\textbf{-1}} & \textcolor{blue}{\textbf{3}} & 0 & 0 & 0 & 0 & 0 & 0 & 0 & 0 \\
0 & 0 & 0 & 0 & \textcolor{red}{\textbf{2}} & \textcolor{red}{\textbf{1}} & \textcolor{blue}{\textbf{3}} & \textcolor{blue}{\textbf{-1}}& 0 & 0 & 0 & 0 & 0 & 0 & 0 & 0 \\
\textcolor{blue}{\textbf{-1}} & \textcolor{blue}{\textbf{3}} & 0 & 0 & 0 & 0 & \textcolor{red}{\textbf{1}} & \textcolor{red}{\textbf{2}} & 0 & 0 & 0 & 0 & 0 & 0 & 0 & 0 \\
\textcolor{blue}{\textbf{3}} & \textcolor{blue}{\textbf{-1}} & 0 & 0 & 0 & 0 & \textcolor{red}{\textbf{2}} & \textcolor{red}{\textbf{1}} & 0 & 0 & 0 & 0 & 0 & 0 & 0 & 0 \\
\hline
0 & 0 & 0 & 0 & 0 & 0 & 0 & 0 & \textcolor{red}{\textbf{1}} & \textcolor{red}{\textbf{2}} & \textcolor{blue}{\textbf{-1}} & \textcolor{blue}{\textbf{3}} & 0 & 0 & 0 & 0 \\
0 & 0 & 0 & 0 & 0 & 0 & 0 & 0 & \textcolor{red}{\textbf{2}} & \textcolor{red}{\textbf{1}} & \textcolor{blue}{\textbf{3}} & \textcolor{blue}{\textbf{-1}} & 0 & 0 & 0 & 0 \\
0 & 0 & 0 & 0 & 0 & 0 & 0 & 0 & 0 & 0 & \textcolor{red}{\textbf{1}} & \textcolor{red}{\textbf{2}} & \textcolor{blue}{\textbf{-1}} & \textcolor{blue}{\textbf{3}} & 0 & 0 \\
0 & 0 & 0 & 0 & 0 & 0 & 0 & 0 & 0 & 0 & \textcolor{red}{\textbf{2}} & \textcolor{red}{\textbf{1}} & \textcolor{blue}{\textbf{3}} & \textcolor{blue}{\textbf{-1}} & 0 & 0 \\
0 & 0 & 0 & 0 & 0 & 0 & 0 & 0 & 0 & 0 & 0 & 0 & \textcolor{red}{\textbf{1}} & \textcolor{red}{\textbf{2}} & \textcolor{blue}{\textbf{-1}} & \textcolor{blue}{\textbf{3}} \\
0 & 0 & 0 & 0 & 0 & 0 & 0 & 0 & 0 & 0 & 0 & 0 & \textcolor{red}{\textbf{2}} & \textcolor{red}{\textbf{1}} & \textcolor{blue}{\textbf{3}} & \textcolor{blue}{\textbf{-1}} \\
0 & 0 & 0 & 0 & 0 & 0 & 0 & 0 & \textcolor{blue}{\textbf{-1}} & \textcolor{blue}{\textbf{3}} & 0 & 0 & 0 & 0 & \textcolor{red}{\textbf{1}} & \textcolor{red}{\textbf{2}}\\
0 & 0 & 0 & 0 & 0 & 0 & 0 & 0 & \textcolor{blue}{\textbf{3}} & \textcolor{blue}{\textbf{-1}} & 0 & 0 & 0 & 0 & \textcolor{red}{\textbf{2}} & \textcolor{red}{\textbf{1}}
\end{array}
\right].
\]

\medskip

\section{Determinant of a block circulant matrices with two parameters}\label{Sec_Blck_Circ_Det}

From Schur complement we have the following Proposition.
\begin{proposition} \label{Proposicion_Schur}
Suppose \(E,F,G\) and \(H\) are respectively \(p\times p\), \(p\times q\), \(q\times p\) and \(q\times q\) matrices, with \(H\) non-singular. If
\[
M =
\left[
\begin{array}{cc}
E & F \\
G & H
\end{array}
\right],
\]
then
\[
\det(M) = \det(H) \det(E - FH^{-1}G).
\]
\end{proposition}

Let \(k\) be a non-negative integer. We denoted with \(\vec{v}_{k}\) a \(k\times 1\) vector and with \(\vec{u}_{k}\) a \(1\times k\) vector given by
\[
(\vec{v}_{k})_{i\, 0} = \left\{
\begin{array}{ll}
1, & \text{if } i = k-1,\\
0,& \text{otherwise},
\end{array}
\right.
\]
and
\[
(\vec{u}_{k})_{0\, j} = \left\{
\begin{array}{ll}
1, & \text{if } j = 0,\\
0,& \text{otherwise}.
\end{array}
\right.
\]

In the proof of the following Lemma we use the notation given in \cite{bapat2014graphs}, Let \(K\) be a \(m\times n\) matrix. If \(S\subseteq [m]\), \(T\subseteq [n]\), then \(K\left[S|T \right]\) will denote the submatrix of \(K\) determined by the rows corresponding to \(S\) and the columns corresponding to \(T\).

\begin{lemma} \label{Lema_Determinante}
Let \(n\) be a non-negative integer. Let \(A\) and \(B\) two non-zero square matrices or order \(r\). If \(AB=BA\) and \(A\) is non-singular, then
\[
\det (I_{n}\otimes A +P_{n}\otimes B) = \det(A)^{n} \det (I_{r}- (-A^{-1}B)^{n}).
\]
\end{lemma}
\begin{proof}
For \(1\leq i \leq n-1\), we defined the following sequences of matrices:
\begin{enumerate}
	\item \(E_{i} := I_{n}\left[[n-i] | [n-i] \right] \otimes A + P_{n}\left[[n-i] | [n-i] \right] \otimes B\);
	\item \(F_{i} := \vec{v}_{n-i} \otimes B\);
	\item \(G_{i} := \vec{u}_{n-i} \otimes (-A)^{-i+1}B^{i}\) and
	\item \(H_{i} := A\);
\end{enumerate}
Note that
\[
I_{n}\otimes A +P_{n}\otimes B =
\left[
\begin{array}{cc}
E_{1} & F_{1} \\
G_{1} & H_{1}
\end{array}
\right]
\text{ and}
\quad
E_{i-1}- F_{i-1}H_{i-1}^{-1}G_{i-1} =
\left[
\begin{array}{cc}
E_{i} & F_{i} \\
G_{i} & H_{i}
\end{array}
\right],
\]
for \(2\leq i \leq n-1\). Therefore by Proposition \ref{Proposicion_Schur}
\[
\det (I_{n}\otimes A +P_{n}\otimes B) = \det(H_{1}) \det (E_{1}-F_{1}H_{1}^{-1}G_{1}),
\]
and
\[
\det(E_{i-1}- F_{i-1}H_{i-1}^{-1}G_{i-1}) = \det(H_{i}) \det (E_{i}- F_{i}H_{i}^{-1}G_{i}),
\]
for \(2\leq i \leq n-1\). Therefore
\begin{align*}
\det (I_{n}\otimes A +P_{n}\otimes B) &= \prod\limits_{i=1}^{n-1} \det(H_{i}) \det (E_{i}- F_{i}H_{i}^{-1}G_{i}) \\
&= (\det(A))^{n-2} \det(E_{n-1}- F_{n-1}H_{n-1}^{-1}G_{n-1}) \\
&= \det(A)^{n} \det (I_{r}- (-A^{-1}B)^{n}).\qedhere
\end{align*}
\end{proof}

Given the matrix \(P_{n}^{s_{1}}\otimes A +P_{n}^{s_{2}}\otimes B = (P_{n}^{s_{1}} \otimes I_{r}) (I_{n}\otimes A +P_{n}^{s_{2}-s_{1}}\otimes B)\). Note that
\[
\det(P_{n}^{s_{1}}\otimes A +P_{n}^{s_{2}}\otimes B) = \det (P_{n}^{s_{1}} \otimes I_{r}) \det(I_{n}\otimes A +P_{n}^{s_{2}-s_{1}}\otimes B),
\]
and \(\det (P_{n}^{s_{1}} \otimes I_{r}) \det(I_{n}\otimes A +P_{n}^{s_{2}-s_{1}}\otimes B)\) is equal to
\begin{equation}\label{Ecuacion_Determinante_Igual}
	\det (P_{n}^{s_{1}} \otimes I_{r}) \det((P_{\sigma_{n,s_{2}-s_{1}}}^{T} \otimes I_{r}) (I_{n}\otimes A +P_{n}^{s_{2}-s_{1}}\otimes B) (P_{\sigma_{n,s_{2}-s_{1}}} \otimes I_{r}))
\end{equation}

\begin{theorem} \label{Teo_Determinante_Bloques}
	Let \(n,s_{1}\) and \(s_{2}\) be non-negative integers such that \(0 \leq s_{1} < s_{2}\leq n-1\) and let \(A\) and \(B\) two non-zero square matrices or order \(r\) such that \(AB=BA\). We have the following:
	\begin{enumerate}
		\item if \(A\) is non-singular, then \(\det (P_{n}^{s_{1}}\otimes A +P_{n}^{s_{2}}\otimes B) =\)
		\[
		\det (P_{n}^{s_{1}})^{r} \det(A)^{n} \left[\det \left( I_{r}- ( -A^{-1}B)^{n \backslash (s_{2}-s_{1})}\right) \right]^{(n,s_{2}-s_{1})}.
		\]
		\item If \(B\)  is non-singular and \(A\) is singular, then \(\det (P_{n}^{s_{1}}\otimes A +P_{n}^{s_{2}}\otimes B) =\)
		\[
		\det (P_{n}^{s_{2}})^{r} \det (B)^{n} \left[\det \left(I_{r}- ( -B^{-1}A)^{n \backslash (n-s_{2}+s_{1})}\right)	\right]^{(n,n-s_{2}+s_{1})}.
		\]
	\end{enumerate}
\end{theorem}
\begin{proof}
Let \(s=s_{2}-s_{1}\). Assume that \(A\) is non-singular. By (\ref{Ecuacion_Determinante_Igual}) and Lemma \ref{Lema_Determinante} we have that \(\det(P_{n}^{s_{1}}\otimes A +P_{n}^{s_{2}}\otimes B)=\)
\begin{align*}
	& = \det (P_{n}^{s_{1}} \otimes I_{r}) \det(I_{n}\otimes A +P_{n}^{s}\otimes B) \\
	& = \det (P_{n}^{s_{1}} \otimes I_{r}) \det((P_{\sigma_{n,s}}^{T} \otimes I_{r}) (I_{n}\otimes A +P_{n}^{s}\otimes B) (P_{\sigma_{n,s}} \otimes I_{r})) \\
	& = \det (P_{n}^{s_{1}} \otimes I_{r}) \det \left( I_{(n,s)} \otimes (I_{n\backslash s} \otimes A + P_{n\backslash s} \otimes B)\right) \\
	& = \det (P_{n}^{s_{1}} \otimes I_{r}) \left( \det(I_{n\backslash s} \otimes A + P_{n\backslash s} \otimes B)\right)^{(n,s)} \\
	& = \det (P_{n}^{s_{1}})^{r} \det(A)^{n} \left(\det \left(I_{r}- ( -A^{-1}B)^{n \backslash s}\right)\right)^{(n,s)}.
\end{align*}
	
Now, if \(B\) is non-singular and \(A\) is singular the proof is analogous to the previous proof.\qedhere
\end{proof}

For example, for \(A\) and \(B\) given in (\ref{Ejem_Circ_Block1}), we have that
\[
\det\left(P_{8} \otimes A + P_{8}^{3} \otimes B \right)= \left(\det(P_{8}) \right)^{2} (\det(A))^{8} \left(\det\left(I_{2} - (-A^{-1}B)^{4} \right) \right)^{2}.
\]
Since,
\begin{equation*}
I_{2} - (-A^{-1}B)^{4} = \frac{81}{20728}
\left[
\begin{array}{rr}
-1296 & -1295 \\
-1295  & -1296
\end{array}
\right],
\quad
\det\left(I_{2} - (-A^{-1}B)^{4} \right) = \frac{1}{2591}
\end{equation*}
and \(\det(A) = -3\). Therefore, \(\det\left(P_{8} \otimes A + P_{8}^{3} \otimes B \right)= (-3)^8 \left(\frac{1}{2591} \right)^{2}\).

\begin{corollary}\label{detCnqs}
Let \(n,s_{1}\) and \(s_{2}\) be non-negative integers such that \(0 \leq s_{1} < s_{2}\leq n-1\) and let \(A\) and \(B\) be two non-zero square matrices or order \(r\) such that \(AB=BA\). \(\det(P_{n}^{s_{1}}\otimes A +P_{n}^{s_{2}}\otimes B) =0\) if and only if
\begin{enumerate}
	\item \(A,B\) are singular, or
	\item \(A\) is non-singular, but \(I_{r}-(- A^ {-1}B)^{(n \backslash (s_{2}-s_{1}))}\) is singular, or
	\item \(A\) is singular, but \(B\) is non-singular and \(I_{r}-( - B^ {-1}A)^{(n \backslash (n-s_{2}+s_{1}))} \) is singular.
\end{enumerate}
\end{corollary}

\medskip

\section{Inverse of block circulant matrices with two parameters}\label{Sec_Blck_Circ_Det_Inv}

\begin{theorem}\label{Teorema_Principal_Bloques}
Let \(n\) be a non-negative integer and let \(A\) and \(B\) be non-zero square matrices of order \(r\), such that \(AB=BA\). If \(I_{n} \otimes A + P_{n} \otimes B\) and \(A^{n} - (-B)^{n}\) are non-singular, then
\begin{equation}
\left(I_{n} \otimes A + P_{n} \otimes B \right)^{-1} = \sum\limits_{i=0}^{n-1} P_{n}^{i} \otimes \left[(-1)^{i} B^{i}A^{n-i-1} \left(A^{n} - (-B)^{n} \right)^{-1} \right]
\end{equation}
\end{theorem}
\begin{proof}
It is just check that
\[
\left(I_{n} \otimes A + P_{n} \otimes B \right) \!\! \left[\sum\limits_{i=0}^{n-1} \!\! P_{n}^{i} \otimes \!\! \left[(-1)^{i} B^{i}A^{n-i-1} \left(A^{n} - (-B)^{n} \right)^{-1} \right]\right] \!\! = \!\! I_{n}\otimes I_{r}.\qedhere
\]
\end{proof}

In order to obtain an explicit formula for the inverse of non-singular circulant matrices of the form \(P_{n}^{s_1} \otimes A + P_{n}^{s_2} \otimes B\) we define
\begin{equation*}
\Omega_{n,s}(i)=(-1)^{\pos_{n,s}(i)}\,\delta_{0,\cycle_{n,s}(i)} \,B^{\pos_{n,s}(i)}\,A^{n\backslash s-1-\pos_{n,s}(i)} \left[A^{n\backslash s} - (-B)^{n\backslash s} \right]^{-1},
\end{equation*}
assuming that \(\left[A^{n\backslash s} - (-B)^{n\backslash s} \right]^{-1}\) exists. Notice that \(\Omega\) satisfies the following properties:
\begin{enumerate}
	\item For \(n>0\), \(\Omega_{n,1}(i)=(-1)^{i} B^{i}A^{n-i-1}\), for all \(i=0,\dots,n-1\).
	\item For \(0<s<n\), \(\Omega_{n\backslash s,1}(i) = \Omega_{n,s}(i\, s)\), for all \(i=0,\dots,(n\backslash s)-1\).
\end{enumerate}

\begin{theorem} \label{Teo_inv_principal_bloques}
Let $n,s_{1},s_{2}$ and \(r\) be a non-negative integers such that \(0\leq s_{1} < s_{2} \leq n-1\). Let $A$ and  $B$ be two square matrices of order \(r\) such that \(AB=BA\). If \(P_{n}^{s_{1}}\otimes A + P_{n}^{s_{2}}\otimes B\) and \(A^{n\backslash s_{2}-s_{1}}-(-B)^{n\backslash s_{2}-s_{1}}\) are non-singular, then
\begin{equation}
\left(P_{n}^{s_{1}}\otimes A + P_{n}^{s_{2}}\otimes B\right)^{-1} = \sum_{i=0}^{n-1} P_{n}^{i} \otimes  \Omega_{n,s_{2}-s_{1}}(i+s_{1}).
\end{equation}
\end{theorem}
\begin{proof}
Let \(s=s_{2}-s_{1}\). By (\ref{Ecuacion_Factorizacion_Bloques}) we have that \(\left(P_{n}^{s_{1}}\otimes A + P_{n}^{s_{2}}\otimes B\right)^{-1}\) is equal to
\[
\left(P_{\sigma_{n,s}} \otimes I_{r} \right) \left[I_{(n,s)} \otimes \left(I_{n\backslash s} \otimes A + P_{n\backslash s} \otimes B \right)\right]^{-1} \left(P_{\sigma_{n,s}}^{T} \otimes I_{r} \right) \left(P_{n}^{n-s_{1}} \otimes I_{r} \right).
\]
Thus, by Theorem \ref{Teorema_Principal_Bloques} and the definition of \(\Omega_{n,s}\), \(\left(P_{n}^{s_{1}}\otimes A + P_{n}^{s_{2}}\otimes B\right)^{-1}\) is
\[
\left(P_{\sigma_{n,s}} \otimes I_{r} \right) \left(\sum_{i=0}^{(n\backslash s)-1} I_{(n,s)} \otimes P_{n\backslash s}^{i} \otimes \Omega_{n\backslash s,1}(i) \right) \left(P_{\sigma_{n,s}}^{T} \otimes I_{r} \right) \left(P_{n}^{n-s_{1}} \otimes I_{r} \right).
\]
Then, by (\ref{Ecuacion_Potencias_P_ns1}) and (\ref{Ecuacion_Potencias_P_ns2})
\begin{align*}
\left(P_{n}^{s_{1}}\otimes A + P_{n}^{s_{2}}\otimes B\right)^{-1} &= \sum_{i=0}^{(n\backslash s)-1} P_{n}^{(i\, s)-s_{1}} \otimes \Omega_{n\backslash s,1}(i) \\
&= \sum_{i=0}^{(n\backslash s)-1} P_{n}^{(i\, s)-s_{1}} \otimes \Omega_{n,s}(i\, s) \\
&= \,\,\,\, \sum_{j=0}^{n-1} P_{n}^{j} \otimes \Omega_{n,s}(j+s_{1}).\qedhere
\end{align*}
\qedhere
\end{proof}

For example, for \(A\) and \(B\) given in (\ref{Ejem_Circ_Block1}), we have that \(\left(P_{8} \otimes A + P_{8}^{3} \otimes B \right)^{-1}\) is equal to
\[
P_{8}\otimes \Omega_{8,2}(2) + P_{8}^{3}\otimes \Omega_{8,2}(4) + P_{8}^{5}\otimes \Omega_{8,2}(6) + P_{8}^{7}\otimes \Omega_{8,2}(0),
\]
where
\begin{equation*}
\begin{array}{rcr}
\Omega_{8,2}(0)= \frac{1}{2519}
\left[
\begin{array}{cc}
2357 & 2206 \\
2206 & 2357
\end{array}
\right],& \phantom{ } & \Omega_{8,2}(2)= -\frac{1}{2519}
\left[
\begin{array}{rr}
1823 & 1219 \\
1219 & 1823
\end{array}
\right], \\
\phantom{ } \\
\Omega_{8,2}(4)= \frac{1}{2519}
\left[
\begin{array}{rr}
202 & -194 \\
-194 & 202
\end{array}
\right], & \text{and} &
\Omega_{8,2}(6)= -\frac{1}{2519}
\left[
\begin{array}{rr}
5508 & -4156 \\
-4156 & 5508
\end{array}
\right].
\end{array}
\end{equation*}

\bigskip

%%%%%%%%%%%%%%%%%%%%%%%%%%%%%%%%%%%%%%%%%%%%%%%%%%%%%%%%%%%%%%%%%%%%%%%%%%%%%%%%%%%%%%%%%%%%%%%%%%%%%%%%%%
%
%%%%%%%%%%%%%%%%%%%%%%%%%%%%%%%%%%%%%%%%%%%%%%%%%%%%%%%%%%%%%%%%%%%%%%%%%%%%%%%%%%%%%%%%%%%%%%%%%%%%%%%%%%
\section*{Acknowledgment}
Even though the text does not reflect it, we carry on many numerical experiments on \cite{sagemath}. They gives us the insight for this paper.

%No one will drive us from the paradise which Sagemathcloud created for us.
%{}\\
%\section*{}
Funding: This work was partially supported by Universidad Nacional de San Luis, grant PROICO 03-0918, MATH AmSud, grant 21-MATH-05, and Agencia I+D+I grant PICT 2020-00549. Denis Videla was partially supported by CONICET and SECyT-UNC.  Andr\'es M. Encinas has been partially supported by Comisi\'{o}n Interministerial de Ciencia y Tecnolog\'{\i}a, grant MTM2017-85996-R.
%\section*{References}

%\bibliographystyle{apalike}
%\bibliography{TAGcitas}

\bibliographystyle{apalike}
\bibliography{TAGcitas}

\end{document}